\documentclass[fullpage,11pt]{article}

\usepackage[utf8]{inputenc}
\usepackage{amsmath,amsthm,amssymb,enumerate,framed}
\usepackage{color}
\usepackage{graphicx}
\usepackage{subfigure}

\newcommand{\N}{\mathbb{N}}
\newcommand{\Z}{\mathbb{Z}}
\newcommand{\R}{\mathbb{R}}
\newcommand{\conv}{\operatorname{conv}}
\newcommand{\setcond}[2]{\left\{ #1 \,:\, #2 \right\}}
\newcommand{\EndMarker}{\hspace*{\fill}~$\square$}
\newcommand{\nul}{\text{nullspace}}

\newcommand{\old}[1]{{}}

\newcommand{\intr}{\operatorname{int}}
\newcommand{\relintr}{\operatorname{relint}}
\newcommand{\rec}{\operatorname{rec}}
\newcommand{\lin}{\operatorname{lin}}
\newcommand{\bd}{\operatorname{bd}}

\newcommand{\copr}{\mathbin{\Diamond}}

\newtheoremstyle{mythmstyle}
	{\topsep}
	{\topsep}
	{\itshape}
	{}
	{\scshape}
	{.}
	{3pt}
	{}
\theoremstyle{mythmstyle}

\newtheorem{nn}{}[section]
\newtheorem{lemma}[nn]{Lemma}
\newtheorem{theorem}[nn]{Theorem}

\newtheorem{prop}[nn]{Proposition}

\newtheorem{claim}{Claim}
\newtheorem{fact}[nn]{Fact}
\newtheorem{question}[nn]{Question}

\theoremstyle{definition}
\newtheorem{REMARK}[nn]{Remark}
\newenvironment{remark}{\begin{REMARK}}{\EndMarker\end{REMARK}}

\newtheoremstyle{itsemicolon}{}{}{\mdseries\rmfamily}{}{\itshape}{:}{ }{}
\newtheoremstyle{itdot}{}{}{\mdseries\rmfamily}{}{\itshape}{:}{ }{}
\theoremstyle{itdot}
\newtheorem*{msc*}{2010 Mathematics Subject Classification}
\newtheorem*{keywords*}{Keywords}

\numberwithin{equation}{section}

\title{Operations that preserve the covering property \\ of the lifting region}
\author{Amitabh Basu and Joe Paat\footnote{Department of Applied Mathematics and Statistics, The Johns Hopkins University, MD, USA. Both authors were supported in part by NSF grant CMMI1452820.}}

\begin{document}

\maketitle

\begin{abstract}
We contribute to the theory for minimal liftings of cut-generating functions. In particular, we give three operations that preserve the so-called covering property of certain structured cut-generating functions. This has the consequence of vastly expanding the set of undominated cut generating functions which can be used computationally, compared to known examples from the literature. The results of this paper are significant generalizations of previous results from the literature on such operations, and also use completely different proof techniques which we feel are more suitable for attacking future research questions in this area. %In particular, we are able to replace a tedious volume argument from previous efforts with a cleaner topological argument which provides a new perspective to analyze the so-called lifting region.
%Finally, we complete the classification of two dimensional $S$-free convex sets when $S$ is the intersection of a translated lattice and a polyhedron, thus settling the covering question in two dimensions for such $S$.
\end{abstract}

\section{Introduction}

%This paper contributes to the theory of {\em lifting} within the {\em cut-generating functions} literature.

\paragraph{Cut-Generating Pairs.} Cut-generating functions are a means to have ``a priori'' formulas for generating cutting planes for general mixed-integer optimization problems. We make this more precise. Let $S$ be a closed subset of $\R^n$ with $0 \not\in S$. Consider the following set, parametrized by matrices $R, P$:

\begin{equation}
	\label{def mixed-int set}
	X_{S}(R,P) := \setcond{(s,y) \in \R_+^k \times \Z_+^\ell }{ Rs + Py\in S },
\end{equation}
where $k, \ell \in \Z_+, n \in \N$, $R \in \R^{n \times k}$ and $P \in \R^{n \times \ell}$ are matrices. Denote the columns of matrices $R$ and $P$ by $r_1,\ldots,r_k$ and $p_1,\ldots,p_{\ell},$ respectively. We allow the possibility that $k =0$ or $\ell = 0$ (but not both). This general model contains as special cases classical optimization models such as (1) Mixed-integer linear programming, (2) Mixed-integer conic and convex programming and (3) Complementarity problems with integer constraints; see~\cite{conforti2013cut}.%A few examples are illustrated below\footnote{Many of these examples were borrowed from~\cite{conforti2013cut}. We felt including them here was better than simply referring the reader to the paper.}.

Given $n \in \N$ and a closed subset $S \subseteq \R^n$ such that $0 \not\in S$, a \emph{cut-generating pair} $(\psi, \pi)$ for $S$ is a pair of functions $\psi, \pi:\R^n \to \R$ such that
\begin{equation}
	\label{psi pi ineq}
	\sum_{i=1}^k\psi(r_i)s_i + \sum_{j=1}^\ell\pi(p_j)y_j \ge 1
\end{equation}
is a valid inequality (also called a \emph{cut}) for the set $X_S(R,P)$ for every choice of $k, \ell \in \Z_+$ and for all matrices $R \in \R^ {n \times k}$ and $P \in \R^ {n \times \ell}$. Cut-generating pairs thus provide cuts that separate $0$ from the set $X_S(R,P)$ of feasible solutions (it is known that $0 \not\in S$ implies $0$ is not in the closed convex hull of $X_S(R,P)$ - see Lemma 2.1 in~\cite{conforti2013cut}). We emphasize that cut-generating pairs depend on $n$ and $S$ and do \emph{not} depend on $k,\ell$, $R$ and $P$. There is a natural partial order on the set of cut generating pairs; namely, $(\psi', \pi') \leq (\psi, \pi)$ if and only if $\psi' \leq \psi$ and $\pi' \leq \pi$. Due to the nonnegativity of $(s,y)$, if $(\psi', \pi') \leq (\psi, \pi)$ then all the cuts obtained from $(\psi, \pi)$ are dominated by the cuts obtained from $(\psi', \pi')$. The minimal elements under this partial ordering are called {\em minimal cut-generating pairs.} It is verified in Proposition~\ref{prop:minimality-lifting} that every valid cut-generating pair is dominated by a minimal one. Thus, one can concentrate on the minimal cut-generating pairs.

\paragraph{Efficient procedures for cut-generating pairs.} Several deep structural results were obtained by Johnson~\cite{johnson} about minimal cut-generating functions for $S$ when $S$ is a translated lattice, i.e., $S= b + \Z^n$ for some $b \in \R^n\setminus \Z^n$. However, a major drawback is that the theory developed is abstract and difficult to use from a computational perspective. A recent approach has been to restrict attention to a specific class of minimal cut-generating pairs for which we can give efficient procedures to compute the values. In particular, given some specific matrices $R, P$, we want to be able to compute the coefficients $\psi(r_i)$ and $\pi(p_j)$ quickly. For this purpose, a relaxed model was proposed~\cite{alww,BorCor,dey2010constrained,conforti2013cut}:

\begin{equation}
	\label{def mixed-int set 2}
	X_{S}(R) := \setcond{s \in \R_+^k }{ Rs \in S }
\end{equation}

A {\em cut-generating function for} $S$ is a function $\psi :\R^n \to \R$ such that \begin{equation}
	\label{psi ineq}
	\sum_{i=1}^k\psi(r_i)s_i \ge 1
\end{equation}
is a valid inequality for the set $X_S(R)$ for every choice of $k\in \Z_+$ and a matrix $R \in \R^ {n \times k}$. For a given $S \subseteq \R^n\setminus\{0\}$, we stress the distinction between a cut-generating {\em pair} for $S$ as defined in~\eqref{psi pi ineq}, and a cut-generating {\em function} for $S$, as defined in~\eqref{psi ineq}. The important distinction is that model~\eqref{def mixed-int set 2} has no integer variables, as opposed to~\eqref{def mixed-int set}.

The notion of a minimal cut-generating function can be analogously defined and it can be shown along the lines of Proposition~\ref{prop:minimality} that all cut-generating functions are dominated by minimal ones. It turns out that for many specially structured $S$, we obtain closed-form formulas for minimal cut-generating functions. This is done via an important connection that was observed between the so-called {\em $S$-free convex sets} and minimal cut-generating functions. Given $S \subseteq \R^n$, a convex set $B$ is called $S$-free if $\intr(B) \cap S = \emptyset$. A {\em maximal} $S$-free set is an $S$-free convex set that is inclusion wise maximal. When $S$ is the intersection of a translated lattice and a polyhedron, i.e., $S = (b + \Z^n) \cap P$ for some vector $b \in \R^n\setminus \Z^n$ and some rational polyhedron $P$, it was shown in~\cite{dey2010constrained,bccz2} that maximal $S$-free sets are polyhedra, and further, a function $\psi:\R^n \to \R$ is a minimal cut-generating function for $S$ if and only if there exists a maximal $S$-free polyhedron $B$ containing the origin in its interior given by \begin{equation}\label{eq:B-desc} B = \{r \in \R^n : a_i\cdot r \leq 1\;\; i\in I\}\end{equation} and \begin{equation}\label{eq:psi-formula} \psi(r) = \max_{i \in I} a_i \cdot r.\end{equation}This connection between maximal $S$-free sets and minimal cut-generating functions was further developed in~\cite{conforti2013cut}. The exciting observation is that we can compute the coefficients $\psi(r_i)$ in~\eqref{psi ineq} very quickly using the formula~\eqref{eq:psi-formula}. The question is: can we find similar formulas for cut-generating {\em pairs} ?

This led Dey and Wolsey~\cite{dw2008} to import the idea of {\em monoidal strengthening} into this context. Monoidal strengthening was a method introduced by Balas and Jeroslow~\cite{baljer} to strengthen cutting planes by using integrality information. This inspired Dey and Wolsey to define the notion of a {\em lifting} of a cut-generating function $\psi$ as any function $\pi : \R^n \to \R$ such that $(\psi, \pi)$ together forms a cut-generating pair. Given a fixed $\psi$ which is a cut-generating function for $S$, the set of all liftings of $\psi$ is partially ordered by pointwise dominance and one can thus define {\em minimal liftings}. Proposition~\ref{prop:minimality-lifting} shows that for any cut-generating function $\psi$ (not necessarily minimal) and any lifting $\pi$ of $\psi$, $\pi$ is dominated by a minimal lifting of $\psi$. It is not hard to observe that if $\psi$ is a minimal cut-generating function, and $\pi$ is a minimal lifting of $\psi$, then $(\psi, \pi)$ is a minimal cut-generating pair. Thus, this becomes an approach to obtain formulas for minimal cut-generating pairs: start with a minimal cut-generating function $\psi$ for $S$ which has an easily computable formula like~\eqref{eq:psi-formula} and find minimal liftings $\pi$ for $\psi$. Hopefully, a formula for $\pi$ can also be derived easily from the formula for $\psi$. This was explicitly proved to be the case under certain conditions in~\cite{averkov-basu-lifting}. This provides evidence to support Dey and Wolsey's method for finding efficient procedures to compute cut-generating pairs.

\begin{remark}
Not every minimal cut-generating pair $(\psi, \pi)$ for $S$ is of the type that $\psi$ is a minimal cut-generating function for $S$ and $\pi$ is a minimal lifting for $\psi$. The Dey and Wolsey approach outlined above focuses on a subset of minimal cut-generating functions so as to be able to compute with these.
\end{remark}

%We next discuss some fundamental facts about minimal liftings.

\paragraph{Unique minimal liftings.} There is some regularity in the structure of minimal liftings. Given an arbitrary $S \subseteq \R^n\setminus\{0\}$ define \begin{equation}\label{eq:W}W_S := \{w\in \R^n : s+\lambda w\in S~, \forall s\in S, \forall \lambda\in \Z\}.\end{equation} Proposition~\ref{prop:periodic} shows that if $\psi$ is a cut-generating function (not necessarily minimal) for $S$, then any minimal lifting $\pi$ is periodic along $W_S$, i.e., $\pi(p + w) = \pi(p)$ for all $p \in \R^n$ and $w \in W_S$.

A central object in the study of minimal liftings is the {\em lifting region} first introduced in~\cite{dw2008}. Let $\psi$ be a minimal cut-generating function for $S$. Define \begin{equation}\label{eq:lifting-region} R_\psi := \{r \in \R^n : \psi(r) = \pi(r) \textrm{ for every minimal lifting }\pi \textrm{ of }\psi\}.\end{equation}

Since every minimal lifting is periodic along $W_S$, if $R_\psi + W_S = \R^n$, then $\psi$ has a {\em unique} minimal lifting. It was shown in~\cite{bcccz} that for the special case when $S$ is a translated lattice, this is a characterization, i.e., $\psi$ has a unique minimal lifting if and only if $R_\psi + W_S = \R^n$. Note that when $S = b + \Z^n$, then $W_S = \Z^n$. In this situation, the question of whether $\psi$ has a unique minimal lifting or not is equivalent to the geometric question of whether $R_\psi + \Z^n = \R^n$, i.e., whether $R_\psi$ covers $\R^n$ by integer translates.

For a general $S$ and a cut-generating function $\psi$ for $S$, if $R_\psi + W_S = \R^n$ then not only do we have a unique minimal lifting, but we can also express this unique minimal lifting compactly in terms of $\psi$:

\begin{equation}\label{eq:formula-for-lifting}
\psi^\ast(r) = \inf_{w \in W_s} \psi(r + w)
\end{equation}

In fact, Proposition~\ref{prop:psi-ast} shows something stronger: $\psi^\ast$ is a minimal lifting if $R_\psi + W_S = \R^n$ (and thus must be the unique minimal lifting) and the infimum in~\eqref{eq:formula-for-lifting} is attained by any $w$ such that $r + w \in R_\psi$. Therefore, if an explicit description for $R_\psi$ can be obtained, then the coefficient $\psi^\ast(p_j)$ for the unique lifting can be computed by finding the $w$ such that $p_j + w \in R_\psi$, and then using the formula for $\psi(p_j+w)$\footnote{For the special case when $S$ is the intersection of a translated lattice and a polyhedron, a proof similar to Proposition 1.1 in~\cite{averkov-basu-lifting} can be used to show that $\psi^\ast(p)$ can be computed in polynomial time when the dimension $n$ is considered fixed, assuming the data is rational.}. A central result in~\cite{bcccz} was to show that when $S$ is the intersection of a translated lattice and a rational polyhedron, $R_\psi$ can be described as the finite union of full dimensional polyhedra, each of which has an explicit inequality description.

In summary, in this approach of using liftings of minimal cut-generating functions to obtain computational efficiency with cut-generating pairs, two questions are of utmost importance:

\begin{itemize}
\item[(i)] For which kinds of sets $S$ can we find explicit descriptions of $R_\psi$ for any minimal cut-generating function $\psi$ for $S$? The most general $S$ that we know the answer to is when $S$ is the intersection of a translated lattice with a rational polyhedron~\cite{bcccz}.
\item[(ii)] For which pairs $S,\psi$, where $\psi$ is a minimal cut-generating function for $S$, is $R_\psi + W_S = \R^n$ ?
\end{itemize}

\paragraph{Statement of Results.} In this paper, we make some progress towards the covering question (ii) stated above for the special case when $S$ is the intersection of a translated lattice and a rational polyhedron, i.e., $S = (b+ \Z^n) \cap P$ where $b\in \R^n\setminus\Z^n$ is a vector, and $P\subseteq \R^n$ is a rational polyhedron. As mentioned earlier, the minimal cut-generating functions for such $S$ are in one-to-one correspondence with maximal $S$-free sets containing the origin in their interior. For any such maximal $S$-free set $B$, we refer to the lifting region $R_\psi$ for the minimal cut-generating function $\psi$ corresponding to $B$ by $R(S,B)$, to emphasize the dependence on $S$ and $B$. We say $R(S,B)$ has the {\em covering property} if $R(S,B) + W_S = \R^n$. When $S$ is clear from the context, we will also say $B$ has the {\em covering property} if $R(S,B)$ has the covering property.

\begin{enumerate}
\item Let $S$ be a translated lattice intersected with a rational polyhedron and let $B$ be a maximal $S$-free set with the origin in its interior. Then $R(S,B) + W_S = \R^n$ if and only if $R(T(S), T(B)) + W_{T(S)} = \R^n$ for all invertible affine transformations $T: \R^n \to \R^n$ such that $T(B)$ also contains the origin in its interior. In other words, the covering property is preserved under invertible affine transformations. This is the content of Theorem~\ref{thm:trans-inv}. This result was first proved for the special case when $S$ is a translated lattice, $B$ is a maximal $S$-free {\em simplicial} polytope and $T$ is a simple translation~\cite{basu2012unique}. In~\cite{averkov-basu-lifting}, the result was generalized to all maximal $S$-free sets when $S$ is a translated lattice and $T$ is a simple translation. Here we generalize the result to all maximal $S$-free sets where $S$ is the intersection of a translated lattice and a rational polyhedron, and allow for $T$ to be any general invertible affine transformation (which, of course, includes simple translations as a special case). Moreover, the proofs in~\cite{basu2012unique} and~\cite{averkov-basu-lifting} are based on volume arguments, whereas our proofs are based on a completely different topological argument. It makes the proof much cleaner, albeit at the expense of using more sophisticated topological tools like the ``Invariance of Domain" theorem. The volume arguments are difficult to extend to tackle more general $S$ sets and general affine transformations $T$, and hence we feel that our approach has a better chance of success for attacking the general covering question (ii) above.

\item In Section~\ref{s:co-product}, we define a binary operation on polyhedra that preserves the covering property. Namely, given two polyhedra $X_1$ and $X_2$, we define the {\em coproduct} $X_1 \copr X_2$ which is a new polyhedron that has nice properties in terms of the lifting region. More precisely, let $n = n_1 + n_2$. For $i\in \{1,2\}$, let $S_i = P_i\cap \Lambda_i$, where $P_i \subseteq \R^{n_i}$ is a rational polyhedron and $\Lambda_i$ is a translated lattice in $\R^{n_i}$. %Then $S_1\times S_2 = (P_1\times P_2)\cap (\Lambda_1\times \Lambda_2)$. Therefore, it is reasonable to speak of $S_1\times S_2$-free sets.
Theorem~\ref{thm:coproduct} shows that if $B_i$ is maximal $S_i$-free such that $R(S_i, B_i)$ has the covering property for $i\in \{1,2\}$, then $\frac{B_1}{\mu}\copr \frac{B_2}{1-\mu}$ is maximal $S_1\times S_2$-free and $R(S_1 \times S_2,\frac{B_1}{\mu}\copr \frac{B_2}{1-\mu})$ has the covering property for every $\mu\in (0,1)$. This is an extremely useful operation to create higher dimensional maximal $S$-free sets with the covering property by ``gluing'' together lower dimensional such sets. This result is a generalization of a result from~\cite{averkov-basu-lifting}, where this was shown when $S$ is a translated lattice, and only lattice-free {\em polytopes} were considered. Here we give the result for more general $S$ sets, and perhaps more interestingly, extend the operation to unbounded $S$-free sets. It is worth noting that a trivial extension of the operation defined in~\cite{averkov-basu-lifting} does not work in the more general setting. The operation defined in this manuscript utilizes prepolars which seems to be the right way to generalize and also leads to simpler proofs compared to~\cite{averkov-basu-lifting}; see Section~\ref{s:co-product} for a discussion.

\item We show that if a sequence of maximal $S$-free sets all having the covering property, converges to a maximal $S$-free set (in a precise mathematical sense), then the ``limit'' set also has the covering property; see Theorem~\ref{thm:LimitOfB}. This result is a generalization of a result from~\cite{averkov-basu-lifting} where this was shown when $S$ is a translated lattice, and only lattice-free polytopes were considered. Here we consider general $S$ sets and allow unbounded $S$-free sets.

%\item For $n=2$, we completely classify all maximal $S$-free sets when $S$ is a translated lattice intersected with a polyhedron. As a result, we immediately characterize all minimal cut-generating functions that have the covering property for $n=2$. This generalizes the original $2$-dimensional result of Dey and Wolsey~\cite{dw2008} where they consider $S$ to be a translated lattice in two dimensions.

\end{enumerate}

The importance of these results in terms of cutting planes is the following. Result 1. above has important practical consequences in generating cutting planes, even in the special case when the affine transformation $T$ is a simple translation. The cutting planes from maximal $S$-free sets for mixed-integer linear programs are useful for cutting off a basic feasible solution of the LP relaxation. Different basic feasible solutions correspond to different $S$ sets, translated by a vector. The translation theorem tells us that if a certain $S$-free set $B$ has good formulas because it has the covering property at a particular basic feasible solution, then $B$ will give rise to good formulas at other basic feasible solutions as well, even though the $S$ set has changed because the basic feasible solution has changed. The situation at the new basic feasible solution can be modeled by translating $S$ and $B$.%Moreover, if certain discrete optimization problems can be modeled with $S$ sets that are affine transformation of truncated lattices, then these theorems will allow the use of maximal $S$-free sets $B$ that are affine transformations of standard maximal lattice-free sets that are known to have the covering property.

Work by Dey and Wolsey~\cite{dey2010constrained,dw2008} has established a ``base set" of maximal $S$-free sets with the covering property in $\R^2$. By iteratively applying the three operations stated in results 1., 2. and 3. above, we can then build a vast (infinite) list of maximal S-free sets (in arbitrarily high dimensions) with the covering property, enlarging this ``base set". Moreover, in~\cite{averkov-basu-lifting}, specific classes of maximal $S$-free polytopes in general dimensions were shown to have the covering property. This contributes to a larger ``base set" from which we can build using the operations in results 1., 2. and 3. Not only does this recover all the previously known sets with the covering property, it vastly expands this list. Earlier, ad hoc families of S-free sets were proven to have the covering property - now we have generic operations to construct infinitely many families. See Section~\ref{sec:examples} for more discussion. %In our opinion, this makes a significant leap in the theory of lifting from previous investigations. 
From a broader perspective, we believe it makes a contribution in the modern thrust on obtaining efficiently computable formulas for computing cutting planes, by giving a much wider class of cut-generating functions whose lifting regions have the covering property. As discussed earlier, this property is central for obtaining computable formulas for minimal liftings.

\section{Preliminaries}

We use $\conv(X)$ to denote the convex hull of a set $X$. We use $\intr(X), \relintr(X), \bd(X)$ to denote the interior, the relative interior and the boundary of a set $X$, respectively. The recession cone and lineality space of a convex set $C$ will be denoted by $\rec(C)$ and $\lin(C)$, respectively. We denote the polar of a convex set $C$ by $C^*$. For sets $A,B$, $A+B := \{a + b: a\in A \;\; b\in B\}$ is the Minkowski sum of sets $A\subseteq \R^n$ and $B\subseteq \R^n$ (when $B$ is a singleton $\{b\}$, we will use $A + b$ to denote $A + \{b\}$). For a set $A\subseteq \R^n$ and $\mu \in \R$, $\mu A := \{\mu t: t\in A\}$. If $A_1\subseteq \R^{n_1}$ and $A_2 \subseteq \R^{n_2}$, then $A_1 \times A_2$ will denote the Cartesian product $\{(a_1, a_2) \in \R^{n_1 + n_2} : a_1 \in A_1, a_2 \in A_2\}$.

A {\em lattice} in $\R^n$ is a subset of $\R^n$ of the form $\{\lambda_1 v_1 + \ldots + \lambda_n v_n: \lambda_i \in \Z\}$ where $v_1, \ldots, v_n$ are linearly independent vectors. When these generating vectors are the standard unit vectors in $\R^n$, we get the standard integer lattice $\Z^n$. A {\em lattice subspace} of a lattice $\Lambda$ is a linear subspace which has a basis composed of vectors from $\Lambda$. We say a set $S$ is a {\em truncated affine lattice} if $S = (b+ \Lambda) \cap C$ for some lattice $\Lambda$ in $\R^n$, some $b \in \R^n\setminus \Lambda$, and some convex set $C \subseteq \R^n$; if $C=\R^n$ we call $S$ an {\em affine lattice} or a {\em translated lattice}. Note that $0 \not\in S$ by construction. In general, for a truncated affine lattice $S$, $\conv(S)$ is not a polyhedron; it may not even be closed~\cite{dey2013some}. If $\conv(S)$ is a polyhedron, we specify further by saying $S$ is a {\em polyhedrally-truncated affine lattice}. In this case, $S = (b + \Lambda) \cap \conv(S)$. The following fact follows from Theorem 5 in~\cite{dey2013some}.

\begin{fact}\label{fact:lin-space-lattice}
If $\conv(S)$ is a polyhedron for a truncated affine lattice $S$, then $\lin(\conv(S))$ is a lattice subspace.
\end{fact}

%Let $\Lambda = \Z^n + t_0$ be some translated lattice in $\R^n$, and let $S = P \cap \Lambda$, where $P \subseteq \R^n$ is some rational polyhedron. {\color{red} Talk about the relaxation model and introduce minimal liftings} Any minimal lifting $\pi$ is periodic along the set $W_S = \{w\in \R^n : s+\lambda w\in S~, \forall s\in S, \forall \lambda\in \Z\}$. {\color{red} Should the proof be in the appendix?}.
\paragraph{Properties of the translation set $W_S$.} Given any arbitrary set $S \subseteq \R^n$, we collect some simple observations about the set $W_S$ defined in~\eqref{eq:W}. Note that $W_S$ is a subgroup of $\R^n$, i.e., $0 \in W_S$, $w_1 + w_2 \in W_S$ for every $w_1, w_2 \in W_S$ and $-w \in W_S$ for every $w\in W_S$. We observe below how $W_S$ changes as certain operations are performed on $S$. The proofs are straightforward and are relegated to the Appendix.

\begin{prop}\label{prop:W-manipulation}
The following are true: %$For sets $W_{S_1}$ and $W_{S_2}$ and $\mu\in \R\setminus\{0\}$, it follows that
\begin{itemize}
%\item[(i)] $W_{\mu S} = \mu W_{S}.$ for all sets $S\subseteq \R^n$ and all $\mu \in \R\setminus\{0\}$.
%\item[(ii)] $W_{S + t} = W_{S}$ for all sets $S\subseteq \R^n$ and $t \in \R^n$.
\item[(i)] $W_{M(S)+m} = MW_S$ for all sets $S\subseteq\R^n$, translation vectors $m\in \R^n$, and invertible linear transformations $M:\R^n\to \R^n$. In particular, $W_{\mu S} = \mu W_{S}$ for all sets $S\subseteq \R^n$ and all $\mu \in \R\setminus\{0\}$. %and $W_{S + t} = W_{S}$ for all sets $S\subseteq \R^n$ and $t \in \R^n$.

\item[(ii)] $W_{S_1\times S_2} = W_{S_1}\times W_{S_2}$ for all sets $S_1 \subseteq \R^{n_1}, S_2 \subseteq \R^{n_2}$. Note that $S_1 \times S_2 \subseteq \R^{n_1 + n_2}$.
\end{itemize}
\end{prop}
%
%\begin{proof}
%\begin{itemize}
%\item[(i)] Note that
%\begin{align*}
%x \in \mu W_s & \iff \frac{x}{\mu} \in W_S\\
% & \iff s+\lambda(\frac{x}{\mu}) \in S, ~\forall s\in S,~\forall \lambda\in\Z\\
% & \iff \mu s+\lambda x\in \mu S, ~\forall s\in S, \forall \lambda\in \Z\\
% & \iff s+\lambda x\in \mu S, ~\forall s\in \mu S, \lambda\in \Z\\
% &\iff x\in W_{\mu S}.
% \end{align*}
%\item[(ii)] Note that
%\begin{align*}
%x \in W_{S+t} &\iff (s+t)+\lambda x \in S+t,~\forall s\in S, ~\forall \lambda\in \Z\\
%&\iff s+\lambda x\in S, ~\forall s\in S, ~\forall \lambda\in\Z\\
%&\iff x\in W_S
%\end{align*}
%\item[(iii)] Note that
%\begin{align*}
%(x_1, x_2) \in W_{S_1\times S_2} & \iff (s_1+\lambda x_1, s_2+\lambda x_2) \in S_1\times S_2, ~\forall (s_1, s_2)\in S_1\times S_2, ~\forall \lambda\in \Z\\
%& \iff s_i+\lambda x_i\in S_i, ~\forall i\in \{1,2\}, ~\forall s_i\in S_i, ~\forall \lambda\in \Z\\
%& \iff (x_1, x_2) \in W_{S_1}\times W_{S_2}.
%\end{align*}
%\end{itemize}
%\end{proof}

When $S$ is a nonempty truncated affine lattice, $W_S$ is a lattice; in particular, we can rewrite $W_S$ as the intersection of $\lin(\conv(S))$ and the lattice $\Lambda$. %We will frequently employ the following characterization.

\begin{prop}\label{prop:W_S-characterization}
Let $S = (b+ \Lambda)\cap C$ be a nonempty truncated affine lattice. Then $W_S = \lin(\conv(S))\cap \Lambda$.
\end{prop}
\begin{proof} Let $w\in W_S$. For each $y\in \conv(S)$, we can write $y = \sum_{i=1}^n \lambda_is_i$ for $\lambda_i\in [0,1]$, $\sum_{i=1}^n \lambda_i = 1$, and $s_i\in S$. It follows that
\begin{equation*}
y+w = \left(\sum_{i=1}^n \lambda_i s_i\right)+w = \sum_{i=1}^n \lambda_i (s_i+w) \in \conv(S),
\end{equation*}
where the inclusion follows from the definition of $W_S$. Since $-w$ is also in $W_S$, this shows that $w \in \lin(\conv(S))$. As $S$ is nonempty, there exists a $s\in S$, and we can write $s=b+z_1$ and $s+w = b+z_2$ for $z_1, z_2\in \Lambda$. Thus, $w = z_2-z_1\in\Lambda$. Hence, $W_S\subseteq \lin(\conv(S))\cap \Lambda$.

Conversely, take $w\in \lin(\conv(S))\cap \Lambda$. For $\lambda\in\Z$ and $s\in S$, it follows that $s+\lambda w \in \conv(S)\subseteq C$. Furthermore, $s=b+z_1$ for $z_1\in\Lambda$, and so $s+\lambda w = (z_1+\lambda w)+b \in \Lambda+b$. Therefore $s+\lambda w\in S$, indicating that $\lin(\conv(S))\cap \Lambda \subseteq W_S$.
\end{proof}

\paragraph{Polyhedrally-truncated affine lattices and an explicit description of the lifting region.}
 Let $S$ be a polyhedrally-truncated affine lattice. Let $B = \{r \in \R^n : a_i\cdot r \leq 1\;\; i\in I\}$ be a maximal $S$-free set with the origin in its interior. %Here the set $I$ indexes the set of facets of $B$.
For each $s \in B \cap S$, define the {\em spindle} $R(s,B)$ in the following way. Let $k \in I$ such that $a_k\cdot s = 1$; such an index exists since $B$ is $S$-free, and therefore, $s$ is on the boundary of $B$. Then

$$R(s,B) := \{r\in \R^n : (a_i - a_k)\cdot r \leq 0,\;\; (a_i - a_k)\cdot (s - r) \leq 0 \;\quad \forall i \in I\}.$$

Define \begin{equation}\label{eq:lifting-desc}R(S,B) := \bigcup_{s \in B\cap S} R(s,B).\end{equation}
It was shown in~\cite{bcccz} that when $S$ is a polyhedrally-truncated affine lattice with $\Lambda = \Z^n$, $R(S,B)$ is the lifting region $R_\psi$ defined in~\eqref{eq:lifting-region} for $\psi$ when $\psi$ is the minimal cut-generating function corresponding to $B$ as defined by~\eqref{eq:psi-formula}. Since every $\psi$ is of this form when $S$ is of this type, this gives an explicit description of the lifting region for any minimal cut-generating function in this situation.

In the rest of the paper, we will consider polyhedrally-truncated affine lattices $S$ and analyze the properties of $R(S,B)$ as defined in~\eqref{eq:lifting-desc} for maximal $S$-free sets $B$ given by~\eqref{eq:B-desc}. We will also sometimes abbreviate $R(s,B)$ to $R(s)$ when the set $B$ is clear from context.

\paragraph{Topological Facts.} We collect here some basic tools from topology that will be used in our analysis.

%\begin{lemma}\label{lemma:patching-open-closed}{\color{red} Give proper citation}. Let $U_\omega \subseteq \R^n, \omega \in \Omega$ be any (possibly infinite) family of open sets, and let $f : \bigcup_{\omega \in \Omega} U_\omega \to \R^n$ be a function such that $f$ restricted to each $U_\omega$ is continuous. Then $f$ is continuous over $\bigcup_{\omega \in \Omega} U_\omega$. The result is also true if $U_\omega$ is a {\em finite} collection of closed sets.
%\end{lemma}

%This result has the following simple corollary.

\begin{lemma}\label{lemma:patching-2}[Theorem 9.4 in~\cite{dugundji1970topology}]
Let $P_\omega \subseteq \R^n, \omega \in \Omega$ be a (possibly infinite) family of polyhedra such that any bounded set intersects only finitely many polyhedra, and $\bigcup_{\omega\in \Omega} P_\omega = \R^n$. Suppose there is a family of functions $A_\omega: P_\omega \to \R^n, \omega \in \Omega$ such that $A_\omega$ is continuous over $P_\omega$ for each $\omega\in \Omega$, and for every pair $\omega_1, \omega_2 \in \Omega$, $A_{\omega_1}(x) = A_{\omega_2}(x)$ for all $x \in P_{\omega_1} \cap P_{\omega_2}$. Then there is a unique, continuous map $A:\R^n \to \R^n$ that equals $A_\omega$ when restricted to $P_\omega$ for each $\omega\in \Omega$.\end{lemma}
%
%\begin{proof}
%For any natural number $n\in \N$, let $B_n$ be the open ball of radius $n$ around the origin. From the hypotheses, for each $n \in \N$, $B_n$ is covered by a finite subfamily of polyhedra $P_\omega$, $\omega \in \Omega_n \subseteq \Omega$ where $\Omega_n$ is finite. By Lemma~\ref{lemma:patching-open-closed}, $T$ is continuous when restricted to $\bigcup_{\omega\in \Omega_n} P_\omega$, and thus on $B_n$ for each $n\in \N$. Since $\bigcup_{n \in \N} B_n = \R^n$, another application of Lemma~\ref{lemma:patching-open-closed} shows that $T$ is continuous over $\R^n$.
%\end{proof}

The following is a deep result in algebraic topology, first proved by Brouwer~\cite{brouwer1911beweis, Dold1995}.

\begin{theorem}\label{thm:invariance-domain}[Invariance of Domain]
If $U$ is an open subset of $\R^n$ and $f : U \to \R^n$ is an injective, continuous map, then $f(U)$ is open and $f$ is a homeomorphism between $U$ and $f(U)$.
\end{theorem}

\paragraph{Structure of the lifting region $R(S,B)$.} Let $S$ be a polyhedrally-truncated affine lattice given as $S = (b + \Lambda) \cap C$ and let $B$ be a maximal $S$-free polyhedron given by~\eqref{eq:B-desc}. We now collect some facts about the lifting region $R(S,B)$ as defined in~\eqref{eq:lifting-desc}.

Define $L_B = \{r \in \R^n: a_i\cdot r = a_j\cdot r, \;\quad \forall i,j \in I\}$. The following is proved in~\cite{bcccz} when $\Lambda = \Z^n$; the result can be seen to hold when $\Lambda$ is a general lattice.

\begin{prop}\label{prop:basic-facts}[Theorem 1 and Proposition 6 in~\cite{bcccz}] Let $S$ be a polyhedrally-truncated affine lattice. $B$ is a full-dimensional maximal $S$-free convex set with $0 \in \intr(B)$ if and only if $B$ is a polyhedron of the form~\eqref{eq:B-desc} with a point from $S$ in the relative interior of every facet.
Further, either $B$ is a halfspace or $\intr(B \cap \conv(S)) \neq \emptyset$. When $\intr(B\cap \conv(S))\neq \emptyset$, the following are true:
\begin{itemize}
\item[(i)] $\rec(B \cap \conv(S)) = \lin(B) \cap \rec(\conv(S)) \subseteq \lin(B) \subseteq L_B$ and $\lin(B)\cap \rec(\conv(S))$ is a cone generated by vectors in $\Lambda$.
\item[(ii)] $\lin(R(s)) = \rec(R(s)) = L_B$ for every $s \in B\cap S$.
\item[(iii)] $R(S,B)$ is a union of finitely many polyhedra.
\end{itemize}
\end{prop}

%By virtue of the above lemma, we may restrict our attention to the case when $B \cap \conv(S)$ is a polytope. Indeed, let $V$ be a lattice subspace of $\R^n$ that is complementary to $\lin(B\cap \conv(S))$, i.e., $V + \lin(B\cap \conv(S)) = \R^n$ and $V \cap \lin(B\cap \conv(S)) = \emptyset$. Then, we can look at $\tilde S = S \cap V$ and $\tilde B = B\cap V$, with which the lifting region $R(S,B)$ can be written as $R(\tilde S, \tilde B) + \lin(B\cap \conv(S))$, and the polytope $\tilde B \cap \conv(\tilde S)$.
%
\begin{prop}\label{prop:L-lin}
%Assume that $B\cap\conv(S)$ is a polytope. Then $L_B\cap \lin(\conv(S)) = \{0\}$.
Suppose $\intr(B \cap \conv(S) \neq \emptyset$. $L_B\cap \lin(\conv(S)) = \lin(B)\cap\lin(\conv(S))$ and $L_B \cap \lin(\conv(S))$ is a lattice subspace of $\Lambda$. Consequently, if $B\cap\conv(S)$ is a polytope, then $L_B\cap \lin(\conv(S)) = \{0\}$.
\end{prop}

\begin{proof}
%First, we write $\conv(S)$ as
%{
%$$\conv(S) = \{x\in\R^d: u_j\cdot x\leq \beta_j ~ j\in H\}.$$
%}%
%From this, it follows that
%{
%$$\rec(B\cap\conv(S)) = \{x\in\R^d:~a_i\cdot x \leq 0~\text{for all }i\in I, ~u_j\cdot x\leq 0~\text{for all }j\in H\}.$$
%}%
%Assume to the contrary that $0\neq r\in L_B\cap \lin(\conv(S))$. Wlog, we may suppose there exists an $i\in I$ such that $a_i\cdot r >0$. However, since $r\in L_B$, $a_j\cdot r >0$ for all $j\in I$. This indicates that $a_j\cdot(-r)<0$ for every $j\in I$. Furthermore, as $r\in \lin(\conv(S))$, it follows that $u_i\cdot (-r)=0$ for all $i\in H$. However, this implies that $-r\in \rec(B\cap\conv(S))$, contradicting that $B\cap\conv(S)$ is a polytope.

Consider $r\in L_B\cap \lin(\conv(S))$. It suffices to show that either $r$ or $-r$ is in $\lin(B)\cap\lin(\conv(S))$. Since, $r \in L_B$, for all $i\in I$, $a_i\cdot r$ have the same sign. If $a_i\cdot r \leq 0$ for all $i\in I$, then $r \in \rec(B)$ and therefore, $r \in \rec(B)\cap \lin(\conv(S)) \subseteq \rec(B) \cap \rec(\conv(S)) = \lin(B) \cap \rec(\conv(S))$ (the equality follows from Proposition~\ref{prop:basic-facts}(i) -- note that since $B$ and $\conv(S)$ are both polyhedra, $\rec(B \cap \conv(S)) = \rec(B) \cap \rec(\conv(S))$). Therefore $r \in \lin(B)$. Since $r\in \lin(\conv(S))$, we thus have $r \in \lin(B)\cap \lin(\conv(S))$. If $a_i\cdot r \geq 0$ for all $i \in I$, then $a_i\cdot (-r) \leq 0$ and so $-r \in \rec(B)$. Repeating the same argument, we obtain $-r \in \lin(B)$. Thus, $-r \in \lin(B)\cap\lin(\conv(S))$.

The assertion that $L_B\cap \lin(\conv(S))$ is a lattice subspace follows from Proposition~\ref{prop:basic-facts} (i), the fact that $\lin(\conv(S))$ is a lattice subspace (Fact~\ref{fact:lin-space-lattice}) and $\lin(B) \cap\lin(\conv(S)) = (\lin(B)\cap \rec(\conv(S)))\cap \lin(\conv(S)).$ \end{proof}
%{\color{red}In the rest of the paper, we will make the assumption that $B \cap \conv(S)$ is a polytope; as discussed above, this is without loss of generality.}

%Presently, we focus our attention to the situation $S=P\cap \Lambda$ for $P$ a rational polyhedron.
%\begin{lemma}\label{lem:L-psi-rational}
%{\color{red}Assume that $\intr(B\cap \conv(S))\neq \emptyset$.} Then $L_B \cap \lin(\conv(S))$ is a rational linear subspace.
%\end{lemma}
%
%\begin{proof} First, note that
%\begin{equation*}
%\lin(B\cap \conv(S)) = \lin(B)\cap \lin(\conv(S)) = (\lin(B)\cap \rec(\conv(S)))\cap \lin(\conv(S)).
%\end{equation*}
%Basu et. al. showed that $\lin(B)\cap \rec(\conv(S))$ is rational {\color{red} (this is when $\intr(B\cap \conv(S))\neq \emptyset$. Must we account for this?)}, and $\lin(\conv(S))$ is rational since $S=P\cap \Lambda$, where $P$ is rational. Hence, $\lin(B\cap \conv(S))$ is rational.
%
%It is sufficient to show that $L_B\cap \lin(\conv(S)) = \lin(B\cap \conv(S))$. The proof of Proposition~\ref{prop:L-lin}, it follows that
%\begin{equation*}
%L_B\cap \lin(\conv(S)) \subseteq \rec(B\cap \conv(S)).
%\end{equation*}
%Therefore, as any $x\in L_B\cap \lin(\conv(S))$ also satisfies $-x\in L_B\cap \lin(\conv(S))$, $L_B\cap \lin(\conv(S))\subseteq \lin(B\cap \conv(S))$.
%
%On the other hand, $\lin(B)\subseteq L_B$ by the definition of $L_B$. Intersecting both sides of this containment by $\lin(\conv(S))$, we have that
%\begin{equation*}
%\lin(B\cap\conv(S)) = \lin(B)\cap \lin(\conv(S)) \subseteq L_B\cap \lin(\conv(S)).
%\end{equation*}
%\end{proof}

\begin{theorem}\label{thm:finite-intersection}
Suppose $\intr(B\cap \conv(S))\neq \emptyset$. A bounded set intersects only finitely many polyhedra from $R(S,B) + W_S$.
\end{theorem}
\begin{proof}

Let $L = L_B \cap\lin(\conv(S))$; $L$ is a lattice subspace by Proposition~\ref{prop:L-lin}. Let $V$ be a lattice subspace such that $V\cap L = \{0\}$ and $(V\cap \Lambda) + (L \cap \Lambda) = \Lambda$ (and so $V+L = \R^n$). Also define $L_1 := V \cap L_B$ and $L_2 := V \cap \lin(\conv(S))$. %By Fact~\ref{fact:lin-space-lattice}, $L_2$ is a lattice subspace.

Note that $L_2\cap L_B = \{0\}$. Indeed,
\begin{equation*}
L_2\cap L_B = (V\cap \lin(\conv(S))\cap L_B = V \cap (L_B\cap \lin(\conv(S))) = V \cap L = \{0\}.
\end{equation*}
%and $L_2\subseteq V$ by definition. This indicates that any $x\in L_2\cap L_B$ is also in $V\cap L$, and so $x$ must be equal to $0$. 
Furthermore, $L_2+L = \lin(\conv(S))$. In order to see this, observe that since $V+L = \R^n$, for every $x\in \lin(\conv(S))$ there exists $v\in V$ and $l\in L$ such that $x=v+l$. Since $v=x-l \in \lin(\conv(S))$, $x\in L_2+L$. Thus $\lin(\conv(S)) \subseteq L_2+L$. The other containment follows from the definitions of $L$ and $L_2$.

We next show that $\lin(\conv(S)) \cap \Lambda = (L_2 \cap \Lambda) + (L\cap \Lambda)$. Consider some $x\in \lin(\conv(S))\cap \Lambda$. Since $x\in \Lambda = (V\cap \Lambda)+(L+\Lambda)$ and $V\cap L = \{0\}$, there exists a unique $v\in V\cap\Lambda$ and $l\in L\cap \Lambda$ such that $x=v+l$. As $x\in \lin(\conv(S)) = L_2+L$ and $L_2\cap L \subseteq L_2 \cap L_B = \{0\}$, there exists a unique $l_2\in L_2$ and $l'\in L$ such that $x=l_2+l'$. By the uniqueness of $v$ and $l$, it follows that $v=l_2$ and $l=l'$. Thus $v\in L_2\cap \Lambda$ and $l\in L\cap \Lambda$. Hence, $\lin(\conv(S)) \cap \Lambda \subseteq (L_2 \cap \Lambda) + (L\cap \Lambda)$. The definitions of $L_2$ and $L$ imply the $\supseteq$ containment.

Let $L' $ be any linear subspace of $\R^n$ containing $L_2$ such that $L' \cap L_B = \{0\}$ and $L' + L_B = \R^n$; such a linear space exists since $L_2 \cap L_B = \{0\}$. Since $L_B$ is the recession cone of each spindle in $R(S,B)$, $R(S,B) = (R(S,B)\cap L')+L_B$ and $R(S,B) \cap L'$ is a finite union of polytopes because $R(S,B)$ is a union of finitely many polyhedra. Moreover, by Proposition~\ref{prop:W_S-characterization},
\begin{align*}
R(S,B) + W_S &= R(S,B)+(\lin(\conv(S))\cap\Lambda) \\&=
((R(S,B)\cap L')+L_B) +((L_2\cap \Lambda)+(L\cap \Lambda))\\
& = (R(S,B)\cap L')+(L_2\cap \Lambda)+(L_B+(L\cap \Lambda))\\
& = (R(S,B)\cap L')+(L_2\cap \Lambda)+L_B,
\end{align*}
where the last equality comes from $L\subseteq L_B$.

 Observe that each bounded set $D$ in $\R^n$ intersects at most as many polyhedra in $R(S,B) + W_S$ as $D + L_B$. Since $L'\cap L_B = \{0\}$, $D + L_B$ intersects the same number of polyhedra in $R(S,B) + W_S$ as $(D+L_B) \cap L'$ intersects polyhedra in $(R(S,B) \cap L') + (L_2 \cap \Lambda)$. The complementary assumption also implies that $(D+L_B) \cap L'$ is a bounded set. Since $R(S,B) \cap L'$ is a finite union of polytopes and $L_2 \cap \Lambda$ is a lattice in $L_2 \subseteq L'$, the bounded set $(D+L_B) \cap L'$ intersects finitely many polytopes in $(R(S,B) \cap L') + (L_2 \cap \Lambda)$.\end{proof}

\begin{lemma}\label{lemma:closed}
Suppose $\intr(B\cap \conv(S))\neq \emptyset$. $R(S,B) + W_S$ is a closed set.
\end{lemma}
\begin{proof}
Let $x \not\in R(S,B) + W_S$. Consider the closed ball $B(x,1)$ of radius one around $x$. By Theorem~\ref{thm:finite-intersection}, $B(x,1)$ intersects only finitely many polyhedra from $R(S,B) + W_S$. The union of these finitely many polyhedra is a closed set and therefore, there exists an open ball $N$ around $x$ that does not intersect any of these polyhedra. But since $N \subseteq B(x,1)$, $N$ does not intersect any other polyhedron from $R(S,B) + W_S$. Hence, the complement of $R(S,B) + W_S$ is open.\end{proof}

\section{The covering property is preserved under affine transformations}
Let $S$ be a polyhedrally-truncated affine lattice and let $B$ be a maximal $S$-free polyhedron given by~\eqref{eq:B-desc}. We want to understand the covering properties of the lifting region when we transform $S$ and $B$ by the same invertible affine transformation. For any linear map $F:\R^n \to \R^n$, $F^*$ will denote its adjoint, i.e., the unique linear map such that $x\cdot F(y) = F^*(x)\cdot y$ for all $x,y \in \R^n$; the adjoint corresponds to taking the transpose of the matrix form of the linear map $F$. To avoid an overuse of parentheses, we will often abbreviate $F(x)$ to $Fx$ wherever this is possible without causing confusion.

\begin{theorem}\label{thm:trans-inv}[Affine Transformation Invariance Theorem] Let $M:\R^n\to \R^n$ be an invertible linear map and $m \in \R^n$. Let $T$ denote the affine transformation $T(\cdot) := M(\cdot) + m$. Suppose that $T(B)$ also contains the origin in its interior (i.e., $a_i\cdot (-M^{-1}m) < 1$ for each $i\in I$). $R(S,B) + W_S = \R^n$ if and only if $R(S',B') + W_{S'} = \R^n$, where $S' = T(S)$ and $B' = T(B)$.
\end{theorem}

Observe that $B' = T(B) = M(B) + m$ is given by $\{r \in \R^n : a'_i\cdot r \leq 1\;\; i\in I\}$, where $$a'_i = \frac{(M^{-1})^*(a_i)}{1+ a_i\cdot M^{-1}(m)} \quad \textrm{for each }\; i\in I.$$ %In the equation above, $a_iM^{-1}$ is vector-matrix multiplication. 
Clearly, $B'$ is a maximal $S'$-free set. For $s' \in B' \cap S'$, the spindle $R(s',B')$ is therefore given by  $$R(s',B') = \{r : (a'_i - a'_k)\cdot r \leq 0,\;\; (a'_i - a'_k)\cdot (s' - r) \leq 0 \;\quad, \forall i \in I\}.$$ The lifting region becomes $R(S',B') = \bigcup_{s' \in B'\cap S'} R(s',B')$. %We would like to show

%If $R(S,B)+W_S = \R^n$, then we say that $B$ has the \textit{covering property}.

\paragraph{Intersections modulo the lattice.} We show an interesting property of different spindles when they intersect after translations by vectors in $W_S$. In particular, two spindles from {\em different} facets cannot intersect in their interiors, and moreover, the ``height" of the common intersection points from the different spindles is the same with respect to the respective facets.

\begin{lemma}\label{lemma:collision}[Collision Lemma] Let $S$ be a polyhedrally-truncated affine lattice and let $B$ be a maximal $S$-free polyhedron given by~\eqref{eq:B-desc}. Let $s_1,s_2\in B\cap S$, and let $i_1, i_2\in I$ be such that $a_{i_1}\cdot s_1=1$ and $a_{i_2}\cdot s_2 = 1$. If $x_1, x_2\in R(S,B)$ are such that $x_1-x_2\in W_S$, $x_1\in R(s_1)$, and $x_2\in R(s_2)$, then $a_{i_1}\cdot x_1 = a_{i_2}\cdot x_2$. Moreover, if $x_1\in \intr(R(s_1))$ and $x_2\in \intr(R(s_2))$, then $a_{i_1}=a_{i_2}$.
\end{lemma}

\begin{proof} If $|I|=1$, then the result is trivial. So suppose $|I|>2$. Assume to the contrary that $a_{i_1}\cdot x_1\neq a_{i_2}\cdot x_2$. Suppose that $a_{i_1}x_1 < a_{i_2}x_2$ (for the proof of the other case, switch the indices in the following argument). Since $x_1-x_2\in W_S$, the point $s_2+(x_1-x_2)$ is contained in $S$. In order to reach a contradiction, it is sufficient to show that $s_2+(x_1-x_2)\in \intr(B)$. We will show this using the definition $B = \{r \in \R^n : a_i\cdot r \leq 1\;\; i\in I\}$.

Take $i\in I$. When $i=i_1$, it follows that
\begin{align*}
a_{i_1}(s_2+(x_1-x_2))& = a_{i_1}(s_2-x_2) +a_{i_1}x_1\\
& \leq a_{i_2}(s_2-x_2)+a_{i_1}x_1 &&\text{Since }x_2\in R(s_2)\\
& = 1-a_{i_2}x_2+a_{i_1}x_1 \\
& < 1-a_{i_1}x_1+a_{i_1}x_1\\
& = 1.
\end{align*}
When $i=i_2$, it follows that
\begin{align*}
a_{i_2}(s_2+(x_1-x_2))& = 1+a_{i_2}x_1-a_{i_2}x_2\\
& < 1+a_{i_2}x_1-a_{i_1}x_1\\
& \leq 1 && \text{Since }x_1\in R(s_1).
\end{align*}
Finally, if $i\not\in \{i_1,i_2\}$, then
\begin{align*}
a_{i}(s_2+(x_1-x_2))& = a_i(s_2-x_2) +a_ix_1\\
& \leq a_{i_2}(s_2-x_2) +a_ix_1 &&\text{Since }x_2 \in R(s_2)\\
& = 1 -a_{i_2}x_2+a_ix_1\\
& < 1-a_{i_1}x_1+a_ix_1\\
& \leq 1 &&\text{Since }x_1\in R(s_1).
\end{align*}
Hence $s_2+(x_1-x_2)\in \intr(B)$, giving a contradiction. Thus $a_{i_1}x_1=a_{i_2}x_2$.

Now suppose that $x_1\in\intr(R(s_1))$ and $x_2\in \intr(R(s_2))$. Assume to the contrary that $a_{i_1}\neq a_{i_2}$. We will again show that $s_2+(x_1-x_2)\in \intr(B)$. Since $a_{i_1}\neq a_{i_2}$ and $x_2\in \intr(R(s_2))$,
\begin{equation*}
a_{i_1}\cdot x_2<a_{i_2}\cdot x_2
\end{equation*}
and
\begin{equation*}
a_{i_1}\cdot(s_2-x_2)<a_{i_2}\cdot(s_2-x_2).
\end{equation*}
Let $i\in I$. If $i=i_1$ then using $a_{i_1}\cdot x_1 = a_{i_2}\cdot x_2$, it follows that
\begin{equation*}
a_{i_1}\cdot(s_2+(x_1-x_2)) = a_{i_1}\cdot(s_2-x_2) + a_{i_1}\cdot x_1< a_{i_2}\cdot(s_2-x_2)+a_{i_1}\cdot x_1= 1-a_{i_2}\cdot x_2+a_{i_1}\cdot x_1= 1.
\end{equation*}
If $i=i_2$ then
\begin{equation*}
a_{i_2}\cdot(s_2+(x_1-x_2)) = a_{i_2}\cdot s_2+ a_{i_2}\cdot x_1-a_{i_2}\cdot x_2
= 1+a_{i_2}\cdot x_1-a_{i_1}\cdot x_1
= 1+(a_{i_2}-a_{i_1})\cdot x_1 < 1,
\end{equation*}
where the inequality comes from $x_1\in \intr(R(s_1))$.
Finally, if $i\not\in \{i_1, i_2\}$ then
\begin{equation*}
a_{i}\cdot (s_2+(x_1-x_2)) = a_i\cdot(s_2-x_2)+a_i\cdot x_1
< a_{i_2}\cdot s_2-a_{i_2}\cdot x_2+a_i\cdot x_1
< 1 - a_{i_2}\cdot x_2+a_{i_1}\cdot x_1
= 1,
\end{equation*}
where the first inequality comes from $x_2\in \intr(R(s_2))$ and the second from $x_1\in \intr(R(s_1))$.
Hence, $s_2+(x_1-x_2)\in \intr(B)$, yielding a contradiction.\end{proof}

%%%%% End proof
%%%%%%%%%%%%%%%%%%%%%%%%%%%%%%%%%%%%%%%%%%%%%%%%

\paragraph{Mapping $R(S,B)+W_S$ onto $R(S',B') + W_{S'}$.} We now describe how one can bijectively map each spindle of $R(S,B)$ onto a spindle in $R(S',B')$ by a linear transformation. We will then be able to map $R(S,B)+W_S$ injectively onto $R(S',B') + W_{S'}$ by a piecewise affine map.

Given a particular polyhedrally-truncated affine lattice $S$, a maximal $S$-free polyhedron $B$ described as~\eqref{eq:B-desc}, and an invertible affine map $M(\cdot)+m$ such that $B' = M(B)+m$ contains the origin in its interior, we define linear transformations $T^{S,B,M,m}_i$ for each $i \in I$ given by $$T_i^{S,B,M,m}(r) = Mr + (a_i\cdot r)m.$$ 

\begin{lemma}\label{lem:inverse}
For every $i \in I$, $T_i^{S,B,M,m}(r) $ is an invertible linear transformation with $T_i^{S',B',M^{-1},-M^{-1}m}(r) = M^{-1}r - (a'_i\cdot r)M^{-1}m$ as its inverse.
\end{lemma}

In the following two lemmas, we drop the superscripts in $T^{S,B,M,m}_i$ to save notational baggage; the lemmas are true for any tuple $S,B,M,m$ such that $S$ is a polyhedrally-truncated affine lattice, $B$ is a maximal $S$-free polyhedron with the origin in its interior, and $M(\cdot)+m$ is an invertible affine transformation such that $M(B)+m$ also contains the origin in its interior.

\begin{lemma}\label{lemma:map}
Let $s \in B\cap S$ and let $k \in I$ be such that $a_k\cdot s = 1$. Then $T_k(R(s,B)) = R(s',B')$, where $s' = Ms+m$.
\end{lemma}
\begin{proof}

We first establish the following claim:

\begin{claim}\label{claim:1} For any $\bar r\in \R^n$ and $i \in I$ such that $(a_i - a_k)\cdot \bar r \leq 0$, we have $(a'_i - a'_k)\cdot T_k(\bar r) \leq 0$.
\end{claim}

\begin{proof}
Consider any such $i\in I$ and $\bar r \in \R^n$ such that $(a_i - a_k)\cdot \bar r \leq 0$ (note that $i = k$ satisfies this hypothesis). We show that $a'_i\cdot T_k(\bar r) \leq a_k\cdot \bar r$. Indeed,
%Suffices to show that $a_i\cdot \bar r = a'_i\cdot T_k(\bar r)$ for all $i \in I$. 
$$\begin{array}{rcl}
a'_i\cdot T_k(\bar r)&=& \frac{(M^{-1})^*(a_i)\cdot(M\bar r + (a_k\cdot \bar r)m)}{1 + a_i\cdot M^{-1}m} \\
& = & \frac{a_i\cdot M^{-1}((M\bar r + (a_k\cdot \bar r)m))}{1 + a_i\cdot M^{-1}m} \\
& = & \frac{a_i\cdot (\bar r + (a_k\cdot \bar r)M^{-1}m)}{1 + a_i\cdot M^{-1}m} \\
& = & \frac{a_i\cdot\bar r + (a_k\cdot \bar r)(a_i\cdot M^{-1}m)}{1 + a_i\cdot M^{-1}m} \\
& \leq & \frac{a_k\cdot\bar r + (a_k\cdot \bar r)(a_i\cdot M^{-1}m)}{1 + a_i\cdot M^{-1}m} \qquad \textrm{Using } (a_i - a_k)\cdot \bar r \leq 0 \\
& = & a_k\cdot \bar r.
\end{array}
$$
Observe that the inequality above holds at equality when $i = k$. Therefore, $(a'_i - a'_k)\cdot T_k(\bar r) \leq (a_k - a_k)\cdot \bar r = 0$.
\end{proof}

Now consider any $\hat r \in R(s,B)$. Therefore, for every $i \in I$ we have that $(a_i - a_k)\cdot \hat r \leq 0$ and $(a_i - a_k)\cdot (s -\hat r) \leq 0$. Observe that $T_k(s - \hat r) = T_k(s) - T_k(\hat r) = (Ms + m) - T_k(\hat r) = s' - T_k(\hat r)$ where the second equality follows from the fact that $a_k\cdot s = 1$. By Claim~\ref{claim:1}, we therefore have $(a'_i - a'_k)\cdot T_k(\hat r) \leq 0$ and $(a'_i - a'_k)\cdot (s' -T_k(\hat r)) \leq 0$. Hence, $T_k(\hat r) \in R(s',B')$. This shows that $T_k(R(s,B)) \subseteq R(s',B')$. Using a similar reasoning with the transformation $T_k^{-1}$, one can show that $T_k^{-1}(R(s',B')) \subseteq R(s,B)$, i.e., $R(s',B') \subseteq T_k(R(s,B))$. This completes the proof. \end{proof}

\begin{lemma}\label{lemma:patching-1}
Let $s_1, s_2 \in B \cap S$ and $w_1, w_2 \in W_S$ such that $(R(s_1) + w_1) \cap (R(s_2) + w_2) \neq \emptyset$ and let $x \in (R(s_1) + w_1) \cap (R(s_2) + w_2)$. Let $i_1, i_2 \in I$ be two indices such that $a_{i_1}\cdot s_1 = 1$ and $a_{i_2}\cdot s_2 = 1$. Then, $T_{i_1}(x - w_1) + Mw_1 = T_{i_2}(x - w_2) + Mw_2$.
\end{lemma}

\begin{proof}
%There exist $r'_1, r''_1 \in \R^n$ and $\mu_1 \in \R$ such that $x - w_1 = r'_1 + \mu_1 r''_1$, and $a_{i_1}\cdot r'_1 = 0$, $a_{i_1}\cdot r''_1 = 1$. Similarly, there exist $r'_2, r''_2 \in \R^n$ and $\mu_2 \in \R$ such that $x - w_2 = r'_2 + \mu_2 r''_2$, and $a_{i_2}\cdot r'_2 = 0$, $a_{i_2}\cdot r''_2 = 1$. Since $(x-w_1) - (x-w_2) \in W_S$, by the Collision Lemma (Lemma~\ref{lemma:collision}) we have $a_{i_1}\cdot (x - w_1) = a_{i_2}\cdot (x - w_2)$; we thus obtain that $\mu_1 = \mu_2$. 

Observe that $$\begin{array}{rcl} T_{i_1}(x - w_1) + Mw_1 & = & M(x-w_1) + (a_{i_1}\cdot (x-w_1))m + Mw_1 \\ & = & Mx  + (a_{i_1}\cdot (x-w_1))m\\ & = & Mx  + (a_{i_2}\cdot (x-w_2))m \qquad \textrm{Using the Collision Lemma (Lemma~\ref{lemma:collision})} \\ &=&  M(x -w_2) + (a_{i_2}\cdot (x-w_2))m + Mw_2 \\ &= & T_{i_2}(x - w_2) + Mw_2.\end{array}$$
%where the second equality comes from Lemma~\ref{lemma:facet-translate}. Similarly, $T_{i_2}(x - w_2) = Mx - Mw_2 + \mu_2 m$. Using the fact that $\mu_1 = \mu_2$, we obtain that $T_{i_1}(x - w_1) + Mw_1 = T_{i_2}(x - w_2) + Mw_2$.

\end{proof}
%
%
%Let $\mu_0, \mu_1, \ldots, \mu_{n-1}$ be such that $x - w_1 = \mu_0 r^{i_1}_0 + \sum_{j=1}^{n-1} \mu_j r^{i_1}_j$ and let $\lambda_0, \lambda_1, \ldots, \lambda_{n-1}$ be such that $x - w_2 = \lambda_0 r^{i_2}_0 + \sum_{j=1}^{n-1} \lambda_j r^{i_2}_j$. Since $a_{i_1}\cdot (x - w_1) = a_{i_2}\cdot (x - w_2)$, we obtain that $\mu_0 = \lambda_0$ (using the facts $a_{i_1}\cdot r^{i_1}_0 = 1$, $a_{i_2}\cdot r^{i_2}_0 = 1$, $a_{i_1}\cdot r^{i_1}_j = 0$ for all $j=1, \ldots, n-1$ and $a_{i_2}\cdot r^{i_2}_j = 0$ for all $j=1, \ldots, n-1$).
%Observe that $$\begin{array}{rcl} T_{i_1}(x - w_1) & = &\mu_0 T_{i_1}(r^{i_1}_0) + \sum_{j=1}^{n-1} \mu_j T_{i_1}(r^{i_1}_j) \\ & = & \mu_0(r^{i_1}_0 + t) +  \sum_{j=1}^{n-1} \mu_j r^{i_1}_j \\ & = & x - w_1 + \mu_0 t\end{array}$$
%Similarly, $T_{i_2}(x - w_2) = x - w_2 + \lambda_0 t$. Using the fact that $\lambda_0 = \mu_0$, we obtain that $T_{i_1}(x - w_1) + w_1 = T_{i_2}(x - w_2) + w_2$.\end{proof}

\paragraph{Proof of Theorem~\ref{thm:trans-inv}.}
\begin{proof}
Note that if $B$ (and $B'$) is a halfspace, then the lifting region is all of $\R^n$, and there is nothing to show. Thus, by Proposition~\ref{prop:basic-facts}, we assume $\intr(B \cap \conv(S)) \neq \emptyset$. It suffices to show that $R(S,B) + W_S = \R^n$ implies $R(S',B') + W_{S'} = \R^n$ because the other direction follows by swapping the roles of $S,B$ and $S', B'$ and using the transformation $M^{-1}(\cdot)-M^{-1}m$ instead of $M(\cdot)+m$.

Assume $R(S,B) + W_S = \R^n$. For every $s \in B\cap S$ and $w \in W_S$, define the polyhedron $P_{s,w} = R(s,B) + w$ and define the map $A_{s,w}: P_{s,w} \to \R^n$ as $A_{s,w}(x) = T^{S,B,M,m}_{k}(x-w) + Mw$, where $k \in I$ is such that $a_k\cdot s = 1$. Since $R(S,B) + W_S = \R^n$, we have

\begin{equation*}
\bigcup_{s \in B\cap S, w \in W_S} P_{s,w} = R(S,B) + W_S = \R^n.
\end{equation*}
By Theorem~\ref{thm:finite-intersection}, any bounded set intersects only finitely many polyhedra from the family $\{P_{s,w} : s \in B\cap S, w \in W_S\}$. Moreover, by Lemma~\ref{lemma:patching-1}, we observe that for any two pairs $s_1, w_1$ and $s_2, w_2$ we have that $A_{s_1, w_1}(x) = A_{s_2, w_2}(x)$ for all $x \in P_{s_1,w_1}\cap P_{s_2,w_2}$. Since each $A_{s,w}$ is an affine map on $P_{s,w}$, Lemma~\ref{lemma:patching-2} shows that there exists a continuous map $A :\R^n \to \R^n$ such that $A$ restricted to $P_{s,w}$ is equal to $A_{s,w}$. Observe that

$$\begin{array}{rcl} R(S',B') + W_{S'} & = & R(S',B') + MW_S \qquad\qquad \textrm{by Proposition~\ref{prop:W-manipulation}(i)} \\
& = & \bigcup_{s' \in B'\cap S', w \in W_S} (R(s',B') + Mw) \\
& = & \bigcup_{s \in B\cap S, w \in W_S} (R(Ms + m, M(B)+m) + Mw) \\
& = & \bigcup_{s \in B\cap S, w \in W_S} A_{s,w}(R(s,B) + w) \\
& = & A( \bigcup_{s \in B\cap S, w \in W_S} (R(s,B) + w) \\
& = & A(R(S,B) + W_S) \\
& = & A(\R^n)
\end{array}
$$
where the fourth equality follows from the definition of $A_{s,w}$ and Lemma~\ref{lemma:map}. If we can show that $A$ is injective, then by Theorem~\ref{thm:invariance-domain}, $A(\R^n) = R(S',B') + W_{S'}$ is open. By Lemma~\ref{lemma:closed}, $R(S',B') + W_{S'}$ is also closed (since $\intr(B \cap \conv(S)) \neq \emptyset$ implies $\intr(B' \cap \conv(S')) \neq \emptyset$ ). Since $\R^n$ is connected, the only nonempty closed and open subset of $\R^n$ is $\R^n$ itself. Thus, $R(S',B') + W_{S'} = \R^n$.

Therefore, it is sufficient to show that $A$ is an injective function. Choose $x,y\in \R^n$ such that $A(x)=A(y)$. Unfolding the definition, this implies that there exists $s_1, s_2\in S\cap B$, $w_1, w_2\in W_S$, and $k_1, k_2\in I$ such that $x\in R(s_1)+w_1$, $y\in R(s_2)+w_2$, and $T^{S,B,M,m}_{k_1}(x-w_1)+Mw_1=T^{S,B,M,m}_{k_2}(y-w_2)+Mw_2 =: z^*$. By Lemma~\ref{lemma:map}, $z^* \in (R(s_1')+Mw_1)\cap (R(s_2')+Mw_2)$, where $s_1'=Ms_1+m$ and $s_2'=Ms_2+m$. Note that $R(s_1')$ and $R(s_2')$ are spindles corresponding to $R(S',B')$, and by Proposition~\ref{prop:W-manipulation}(i), $Mw_1, Mw_2\in W_{S'}$.  Therefore, by Lemma~\ref{lemma:patching-1}, $T^{S,B,M^{-1},-M^{-1}m}_{k_1}(z^*-Mw_1)+w_1=T^{S,B,M^{-1},-M^{-1}m}_{k_2}(z^*-Mw_2)+w_2$. By Lemma~\ref{lem:inverse}, $T^{S',B', M^{-1}, -M^{-1}m}_i$ is the inverse of $T^{S,B,M,m}_i$ for each $i\in I$, and so we have $T^{S,B,M^{-1},-M^{-1}m}_{k_1}(z^*-Mw_1)+w_1 = T^{S,B,M^{-1},-M^{-1}m}_{k_1}\left((T^{S,B,M,m}_{k_1}(x-w_1)\right)+w_1 = x.$
Similarly, $T^{S',B',M^{-1}, -M^{-1}m}_{k_2}(z^*-Mw_2)+w_2=y$. Hence $x=y$ and $A$ is injective.  \end{proof}

\section{Generation of S-free sets using coproducts}\label{s:co-product}

Here we display how the covering property is preserved under the so-called {\em coproduct} operation. Given a convex set $C\subseteq \R^n$ containing the origin in its interior, we say $X \subseteq \R^n$ is a {\em prepolar} of $C$ if $X^* = C$, i.e., $C$ is the polar of $X$. We use the notation $C^\bullet$ to denote the smallest prepolar of $C$ with respect to set inclusion. To the best of our knowledge, this concept was first introduced in~\cite{conforti2013cut}, where the authors establish that there is a {\em unique} smallest prepolar. Given closed convex sets $C_1 \subseteq \R^{n_1}, C_2 \subseteq \R^{n_2}$ (possibly unbounded) such that each contains the origin in its interior, define the {\em coproduct of $C_1, C_2$} in $\R^{n_1+ n_2}$ as

\begin{equation}\label{eq:coprod-ineq-2}
C_1\copr C_2 := (C_1^\bullet \times C_2^\bullet)^*.
\end{equation}

If the convex sets are polyhedra given using inequality descriptions, $P_1 = \{x\in\R^{n_1}: a^1_ix \leq 1,~ \forall i\in I_1\}$ and  $P_2 = \{x\in\R^{n_2}: a^2_jx\leq 1, ~\forall j\in I_2\}$, then
\begin{equation}\label{eq:coprod-ineq}
P_1\copr P_2 = \{(x,y) \in\R^{n_1}\times \R^{n_2} : (a^1_i, a^2_j)\cdot (x,y)\leq 1, ~\forall i\in I_1, ~\forall j\in I_2\}.
\end{equation}

The coproduct definition is motivated as a dual operation to Cartesian products: if $P_1$ and $P_2$ are polytopes containing the origin in their interiors, then $(P_1 \times P_2)^* = P_1^* \copr P_2^*$. In this case, our definition specializes to the operation known as the {\em free sum} in polytope theory~\cite[p. 250]{MR1730169}: $P_1 \copr P_2 := \conv ( P_1 \times \{o_2\} \cup \{o_1\} \times P_2)$. The free sum operation was utilized in Section 4 of~\cite{averkov-basu-lifting}, where the operation was also called the coproduct following a suggestion by Peter McMullen. Since our construction is a generalization to the case where $P_1, P_2$ are allowed to be unbounded polyhedra, we retain the terminology of coproduct. If we take closed convex hulls, then the free sum operation can be extended to unbounded sets. Using this extension for unbounded sets, the free sum operation is {\em different} from the coproduct operation defined in~\eqref{eq:coprod-ineq-2}  -- consider the coproduct and free sum of a ray in $\R$ containing the origin and an interval in $\R$ containing the origin. In fact, $\overline\conv (C_1 \times \{o_2\} \cup \{o_1\} \times C_2) = (C_1^* \times C_2^*)^*$ and the second term is different from $(C_1^\bullet \times C_2^\bullet)^*$ when $C_1$ or $C_2$ are unbounded. One can check that parts (ii) and (iii) of Theorem~\ref{thm:coproduct} below fail to hold if one uses $\overline\conv (C_1 \times \{o_2\} \cup \{o_1\} \times C_2) = (C_1^* \times C_2^*)^*$ as the generalization of the operation defined in~\cite{averkov-basu-lifting}.

If each $a^1_i$, $i\in I_1$, gives a facet-defining inequality for $P_1$ and each $a^2_j$, $j\in I_2$, gives a facet-defining inequality for $P_2$, then each inequality in the description in~\eqref{eq:coprod-ineq} is facet-defining. This follows from the fact that each $a^1_i$, $i\in I_1$ is a vertex of $P_1^*$, and similarly, each $a^2_j$, $j\in I_2$ is a vertex of $P_2^*$, and so $(a^1_i, a^2_j)$, $i \in I_1, j \in I_2$ is a vertex of $P_1^* \times P_2^* = \overline\conv(P_1^\bullet \times P_2^\bullet)$.

% We include a proof here that if $P$ is a polyhedron containing the origin in its interior, then $a \in P^*$ is a vertex of $P^*$ if and only if $a\cdot x \leq 1$ is a facet defining inequality for $P$.
%
%If $a$ is not a vertex of $P^*$, then there exist $a_1 \neq a_2$ in $P^*$ such that $a = \frac{a_1 + a_2}{2}$. This implies that for every $x$ such that $a\cdot x = 1$, we have $a_i\cdot x =1$ for $i=1,2$. Thus, the set $\{x : a\cdot x = 1\} \subseteq \{x : a_i \cdot x = 1 \;\; i=1,2\}$ which cannot have dimension $n-1$ since $a_1 \neq a_2$ (if they are scalings of each other, then the set is empty).
%
% If $a\cdot x \leq 1$ is not a facet, then by Farkas' lemma there must exist $a_1, \ldots, a_k \in P^*$ and $\lambda_i \geq 0$, $i=1, \ldots, k$ such that $\sum \lambda_i a_i = a$ and $\sum \lambda_i \leq 1$. Moreover, since $a\cdot x \leq 1$ is not a facet, there must exist a decomposition like this with at least two of the $\lambda_i$'s are strictly positive.  Thus, $a \in \conv(a_1, \ldots, a_k, 0)$ and $a$ is a strict convex combination of these vectors. Thus $a$ cannot be a vertex of $P^*$.
%
For $h\in \{1,2\}$, let $S_h = (b_h + \Lambda_h) \cap P_h$ be two polyhedrally-truncated affine lattices in $\R^{n_h}$ where $P_h = \conv(S_h)$ is a polyhedron. Then $S_1\times S_2 = ((b_1, b_2) + (\Lambda_1\times \Lambda_2)) \cap (P_1\times P_2)$ is also a polyhedrally-truncated affine lattice in $\R^{n_1 + n_2}$. The following result creates $S_1\times S_2$-free sets from $S_1$-free sets and $S_2$-free sets.

\begin{theorem}\label{thm:coproduct}
For $h\in \{1,2\}$, let $B_h\subseteq \R^{n_h}$ be given by facet defining inequalities $\{x\in\R^{n_h}: a^h_ix \leq 1,~ \forall i\in I_h\}$ and let $S_h$ be polyhedrally-truncated affine lattices. Let $\mu \in (0,1)$. Then
\begin{itemize}
\item[(i)] If $B_h$ is $S_h$-free for $h\in \{1,2\}$, then $\frac{B_1}{\mu}\copr \frac{B_2}{1-\mu}$ is $S_1\times S_2$-free.
\item[(ii)] If $B_h$ is maximal $S_h$-free for $h\in \{1,2\}$, then $\frac{B_1}{\mu}\copr \frac{B_2}{1-\mu}$ is maximal $S_1\times S_2$-free.
\item[(iii)] If $B_h$ is maximal $S_h$-free with the covering property for $h\in \{1,2\}$, then $\frac{B_1}{\mu}\copr \frac{B_2}{1-\mu}$ is maximal $S_1\times S_2$-free with the covering property.
\end{itemize}
\end{theorem}

\begin{proof}
\begin{itemize}
\item[(i)] %We may assume that each inequality in the definition of $B_1$ and $B_2$ is facet defining. Therefore, by the comment above, each inequality in $\frac{B_1}{\mu}\copr \frac{B_2}{1-\mu}$ is facet defining. 
Note that
    \begin{equation*}
    \frac{B_1}{\mu} \copr \frac{B_2}{1-\mu} = \{(x_1, x_2)\in\R^{n_1+n_2}: (\mu a_i^1, (1-\mu)a_j^2)\cdot(x_1, x_2)\leq 1, ~\forall i\in I_1,~ j\in I_2\}.
    \end{equation*}
    Let $(s_1, s_2)\in S_1\times S_2$. As $B_1$ is $S_1$ free, there exists an $\overline{i}\in I_1$ such that $a_{\overline{i}}^1\cdot s_1 \geq 1$. Similarly, there is a $\overline{j}\in I_2$ such that $a_{\overline{j}}^2\cdot s_2 \geq 1$. This implies that $(\mu a_{\overline{i}}^1, (1-\mu)a_{\overline{j}}^2)\cdot(s_1, s_2)\geq 1$. Hence, $(s_1, s_2)\not\in \intr(\frac{B_1}{\mu}\copr \frac{B_2}{1-\mu})$.
\item[(ii)] From part (i) and Proposition~\ref{prop:basic-facts}, it is suffices to show that every facet of $\frac{B_1}{\mu}\copr \frac{B_2}{1-\mu}$ contains an $S_1\times S_2$ point in its relative interior. As noted earlier, each inequality in $\frac{B_1}{\mu}\copr \frac{B_2}{1-\mu}$, as given by~\eqref{eq:coprod-ineq}, is facet defining. Consider the facet defined by $(\mu a_{\overline{i}}^1,(1-\mu) a_{\overline{j}}^2)$. Since $a_{\overline{i}}^1$ defines a facet in $B_1$, there exists some $s_1\in S_1$ such that $a_{\overline{i}}^1\cdot s_1 = 1$ and $a_i^1\cdot s_1<1$ for $i \in I_1$ with $i \neq \overline{i}$. Similarly, there exists a $s_2\in S_2$ such that $a_{\overline{j}}^2\cdot s_2 = 1$ and $a_j^2\cdot s_2 <1$ for $ j\in I_2$ with $j \neq \overline{j}$. It follows that $(\mu a_{\overline{i}}^1,(1-\mu) a_{\overline{j}}^2)\cdot (s_1, s_2) = 1$ and $(\mu a_i^1,(1-\mu) a_j^2)\cdot (s_1, s_2) < 1$ for $(i,j)\neq (\overline{i}, \overline{j})$. Hence $(s_1, s_2)$ is in the relative interior of the facet defined by $(\mu a_{\overline{i}}^1,(1-\mu) a_{\overline{j}}^2)$.
\item[(iii)] In order to show that $\frac{B_1}{\mu}\copr \frac{B_2}{1-\mu}$ has the covering property, it is sufficient to show that $\R^{n_1+n_2}\subseteq R+W_{S_1\times S_2}$, where $R = R\left(S_1 \times S_2, \frac{B_1}{\mu}\copr \frac{B_s}{1-\mu}\right)$.

%    As shown in (ii), if $s_1\in B_1$ is in the relative interior of the facet defined by $a_{\overline{i}}^1$ and $s_2\in B_2$ is in the relative interior of the facet defined by $a_{\overline{j}}^2$, then $(s_1, s_2)$ is in the relative interior of the facet defined by $(\mu a_{\overline{i}}^1, (1-\mu) a_{\overline{j}}^2)$.

Consider $s_1 \in B_1 \cap S_1$ and $s_2 \in B_2 \cap S_2$. Let $\bar i \in I_1$ index the facet of $B_1$ containing $s_1$ and $\bar j \in I_2$ index the facet of $B_2$ containing $s_2$. Calculations similar to parts (i) and (ii) above show that $(s_1, s_2)$ lies on the facet of $\frac{B_1}{\mu}\copr \frac{B_s}{1-\mu}$ indexed by $(\bar i, \bar j)$. We claim that the spindle $R(s_1, s_2)$ corresponding to $\frac{B_1}{\mu}\copr \frac{B_2}{1-\mu}$ contains the Cartesian product $R(s_1)\times R(s_2)$. Indeed, a vector $(x_1, x_2)\in R(s_1, s_2)$ if and only if
    \begin{equation*}
    \left((\mu a_i^1, (1-\mu)a_j^2)-(\mu a_{\overline{i}}^1,(1-\mu)a_{\overline{j}}^2)\right)\cdot(x_1, x_2) \leq 0, ~\forall i\in I_1, ~j\in I_2
    \end{equation*}
    and
    \begin{equation*}
    \left((\mu a_i^1, (1-\mu)a_j^2)-(\mu a_{\overline{i}}^1,(1-\mu)a_{\overline{j}}^2)\right)\cdot((s_1,s_2)-(x_1, x_2)) \leq 0, ~\forall i\in I_1, ~j\in I_2.
    \end{equation*}
    where $(\bar i, \bar j)$ indexes the facet containing $(s_1, s_2)$.
    Using the definition of $R(s_i)$, the latter condition follows since $x_1\in R(s_1)$, $x_2\in R(s_2)$, and $\mu\in (0,1)$. %Therefore, as every pairing of points $(s_1, s_2)\in (B_1\times B_2)\cap (S_1\times S_2)$ yields a spindle in the coproduct, we get the containment $R(S_1, B_1)\times R(S_2, B_2)\subseteq R(S_1\times S_2,\frac{B_1}{\mu}\copr \frac{B_2}{1-\mu})$.
    Therefore, we get the containment $R(S_1, B_1)\times R(S_2, B_2)\subseteq R(S_1\times S_2,\frac{B_1}{\mu}\copr \frac{B_2}{1-\mu})$.

    From Proposition \ref{prop:W-manipulation}, we have that $W_{S_1\times S_2} = W_{S_1}\times W_{S_2}$. Since $B_1$ and $B_2$ are assumed to each have the covering property, it follows that
    \begin{align*}
    \R^{n_1+n_2} & = \R^{n_1} \times \R^{n_2} \\
    & = (R(S_1,B_1)+W_{S_1})\times (R(S_2,B_2)+W_{S_2})\\
     & = (R(S_1,B_1)\times R(S_2,B_2))+(W_{S_1}\times W_{S_2})\\
     & \subseteq R(S_1 \times S_2,\frac{B_1}{\mu}\copr \frac{B_2}{1-\mu}) + (W_{S_1}\times W_{S_2})\\
     & = R(S_1 \times S_2,\frac{B_1}{\mu}\copr \frac{B_2}{1-\mu}) + W_{S_1 \times S_2}.
    \end{align*}
    Hence, $\frac{B_1}{\mu}\copr \frac{B_2}{1-\mu}$ has the covering property.
\end{itemize}\end{proof}

Note that (i) above holds for general closed sets $S_h$ and $S_h$-free sets $B_h$.

\section{Limits of maximal S-free sets with the covering property}

%Let $I$ be a finite index set.
Let $m \in \N$ be fixed. For $t\in \N$, let $A^t \in \R^{m \times n}$ be a sequence of matrices and $b^t \in \R^m$ be a sequence of vectors such that $A^t \to A$ and $b^t \to b$ (both convergences are entrywise, i.e., convergence in the standard topology). Let $P_t = \{x\in\R^n: A^t\cdot x\leq b^t\}$ be the sequence of polyhedra defined $A^t, b^t$. We say that $P_t$ converges to the polyhedron $P :=  \{x\in\R^n: A\cdot x\leq b\}$ and we write this as $P_t\to P$. We make some observations about this convergence.

\begin{prop}\label{prop:NullConvergence}
Let $\{A^t\}_{t=1}^\infty$ be a sequence of matrices in $\R^{n\times m}$ converging entrywise to a matrix $A$. If the dimension of the nullspace of $A^t$ is fixed for all $t$, say with value $k$, then the dimension of the nullspace of $A$ is at least $k$.
\end{prop}

\begin{proof} If $k=0$, then the result is trivial. So assume that $k>0$. For each value of $t$, there exists orthonormal vectors $\{v^t_1, v^t_2, \dots, v^t_k\}$ that span the $\nul(A^t)$. Let $V^t\in \R^{n\times k}$ be the matrix with $v^t_i$ as the $i$-th column. As each $v^t_i$ is bounded in $\R^n$, $V^t$ is bounded in $\R^{n\times k}$. Hence, we may extract a convergent subsequence converging to a matrix $V$. By continuity of the inner product of vectors, the columns of $V$ are orthornormal and $AV=0$. Hence, $\dim(\nul(A))\geq k$.
\end{proof}

\begin{prop}\label{prop:PolytopeBounded} Suppose that $\{P_t\}$ is a sequence of polyhedra defined by $P_t = \{x\in\R^n:A^t\cdot x\leq b^t\}$. If $P_t\to P$, where $P$ is a polytope, and $P\cap P_t\neq\emptyset$ for each $t$, then there exists $M\in\R$ such that $P\subseteq [-M,M]^n$ and the sequence $\{P_t\}$ is eventually contained in $[-M,M]^n$. Consequently, the polyhedra in the sequence eventually become polytopes.
\end{prop}

\begin{proof} It suffices to show that for every $\epsilon>0$, there exists a sufficiently large $t$, $P_t\subseteq P+\epsilon B(0,1)$, where $B(0,1)$ is the unit ball.

Assume to the contrary that this is not the case. This indicates that there exists a subsequence of points $\{x_{t_k}\}_{k=1}^\infty$ such that $x_{t_k}\in P_{t_k}\setminus (P+\epsilon B(0,1))$. For each $k \in \N$, there exists some $z_k\in P_{t_k}\cap P$ since $P_{t_k}\cap P\neq \emptyset$. Since the distance function is continuous, there exists some point $y_k\in P_{t_k}\setminus P$ on the line segment $[x_k,z_k]$ such that $y_k\in Y:=\{x\in \R^n: \epsilon/2 \leq d(P, x) \leq \epsilon\}$. Consider the sequence $\{y_k\}_{k=1}^\infty$. Note $Y$ is compact since $P$ is a polytope. Therefore, there exists a subsequence $\{y_{k_j}\}$ of $\{y_k\}$ such that $y_{k_j}\to y$ in $Y$. Let $A^t \to A$ and $b^t \to b$. Since $y\not\in P$, there exists some $i^*\in \{1, \ldots, m\}$ such that $a_{i^*}\cdot y > b_{i^*}$ where $a_{i^*}$ is the row of $A$ indexed by $i^*$ and $b_{i^*}$ is the $i^*$-th component of $b$. However, this implies that
\begin{equation*}
a_{i^*}\cdot y > b_{i^*} = \lim_{j\to\infty} b^{t_{k_j}}_{i^*} \geq \lim_{j\to\infty} a^{t_{k_j}}_{i^*}\cdot y_{k_j} = a_{i^*}\cdot y,
\end{equation*} where $a^{t_{k_j}}_{i^*}$ is the row of $A^{t_{k_j}}$ indexed by $i^*$ and $b^{t_{k_j}}_{i^*}$ is the $i^*$-th component of $b^{t_{k_j}}$. Thus, we reach a contradiction.\end{proof}
%where the last equality comes from the continuity of the inner product.
%which is a contradiction.\end{proof}

%$P_t = \{x\in\R^n: a_i^t\cdot x\leq b^t_i, ~\forall i\in I\}$ be a collection of polyhedra. Note that the index set $I$ is fixed over the sequence $\{B_t\}_{t=1}^\infty$. For each $i\in I$, suppose that $a_i^t\to a_i$. For this section, we consider the limit of the sequence $B_t$ to be the set $B = \{x\in\R^n: a_i\cdot x\leq 1\}$. We denote this by $B_t\to B$.

\begin{prop}\label{prop:full-polytope}
Suppose that $\{P_t\}$ is a sequence of polyhedra defined by $P_t = \{x\in\R^n:A^t\cdot x\leq b^t\}$. If $P_t\to P$ and $x \in \intr(P)$, then there exists $t_0 \in \N$ such that $x \in \intr(P_t)$ for all $t \geq t_0$.
\end{prop}
\begin{proof} As $x \in \intr(P)$, there exists $\delta>0$ such that $\delta\mathbf{1} < b - Ax$, where $\mathbf{1}\in \R^m$ is the vector of all ones. Since $A^t \to A$ and $b^t \to b$, we have that $b^t - A^tx \to b - Ax$ and thus there exists $t_0 \in \N$ such that $b^t - A^tx \geq \delta\mathbf{1}$ for all $t \geq t_0$ and so $x \in \intr(P_t)$ for all $t \geq t_0$.\end{proof}

We next build some tools to prove our main result of this section, Theorem~\ref{thm:LimitOfB}, which is about limits of maximal $S$-free sets that possess the covering property. For the rest of this section, we consider an arbitrary polyhedrally-truncated affine lattice $S$. If $B$ be a maximal $S$-free polyhedron given by~\eqref{eq:B-desc}, recall the definition $L_B = \{r \in \R^n: a_i\cdot r = a_j\cdot r, \;\quad \forall i,j \in I\}$.   %We denote this by $B_t\to B$.

\begin{prop}\label{prop:ComplementarySpaces}
Let $B$ be a maximal $S$-free set and assume that $B\cap\conv(S)$ is a full-dimensional polytope. If $B$ has the covering property, then $L_B+\lin(\conv(S)) = \R^n$.
\end{prop}

\begin{proof}
Assume to the contrary that $L_B+\lin(\conv(S))\neq\R^n$. We claim that $R(S,B)+W_S\neq \mathbb{R}^n$, yielding our contradiction.

%Note that both $L_B$ and $\lin(\conv(S))$ are linear subspaces of $\R^n$.
Since $L_B+\lin(\conv(S))\neq\R^n$, we may choose a subspace $M$ of $\R^n$ such that $\lin(\conv(S))\subsetneq M$ and $L_B+M = \R^n$. Furthermore, as a consequence of Proposition \ref{prop:L-lin}, we may choose $M$ so that $M\cap L_B = \{0\}$. Define $M(S,B) := R(S,B)\cap M$. Note that $M(S,B)$ is compact as the recession cone of every spindle in $R(S,B)$ is $L_B$. Also, $R(S,B) = L_B+M(S,B)$. As $M(S,B)$ is compact, $\lin(\conv(S))+M(S,B)\subsetneq M$. Therefore, using Proposition~\ref{prop:W_S-characterization},
\begin{equation*}
R(S,B) + W_S = L_B+M(S,B)+(\lin(\conv(S))\cap \Lambda) \subsetneq L_B+M =\R^n.
\end{equation*}\end{proof}

\begin{prop}\label{prop:CoveringEquivalence} Suppose $B$ is a maximal $S$-free set such that $L_B+\lin(\conv(S)) = \R^n$ and $L_B \cap \lin(\conv(S)) = \{0\}$. Define $M:=R(S,B)\cap \lin(\conv(S))$. Then the covering property $R(S,B)+W_S = \R^n$ is equivalent to $M+W_S=\lin(\conv(S))$.% In particular, Proposition~\ref{prop:ComplementarySpaces} implies the condition $L_B+\lin(\conv(S)) = \R^n$ holds when $B$ has the covering property and $B\cap \conv(S)$ is a polytope.
\end{prop}

\begin{proof}
Suppose $R(S,B)+W_S = \R^n$. Intersecting both sides by $\lin(\conv(S))$, we see that $(R(S,B)+W_S)\cap \lin(\conv(S)) = \lin(\conv(S))$. It is sufficient to show that $(R(S,B)+W_S)\cap \lin(\conv(S)) = M+W_S$. Take $r+w\in (R(S,B)+W_S)\cap \lin(\conv(S))$ for $r\in R(S,B)$ and $w\in W_S\subseteq \lin(\conv(S))$ by Proposition~\ref{prop:W_S-characterization}. As $r+w\in \lin(\conv(S))$, $r\in \lin(\conv(S))$. Thus, $r\in R(S,B)\cap \lin(\conv(S))$. Hence, $r\in M$ and $(R(S,B)+W_S)\cap \lin(\conv(S)) \subseteq M+W_S$. The other inclusion follows immediately from $W_S\subseteq \lin(\conv(S))$.

Now suppose that $M+W_S = \lin(\conv(S))$ and take $x\in \mathbb{R}^n$. Since $L_B$ and $\lin(\conv(S))$ are complementary spaces, there exists $l\in L_B$ and $s\in \lin(\conv(S))$ such that $x = l+s$. By our assumption, there is an $m\in M$ and $w\in W_S$ so that $x = l+s = l+(m+w) = (l+m)+w$. Since $m\in R(S,B)$, $m$ is contained in some spindle belonging to $R(S,B)$. However, $L_B$ is the lineality space of each spindle. Hence, $l+m\in R(S,B)$. This shows that $\R^n\subseteq R(S,B)+W_S$. The other inclusion follows as $\R^n$ is the ambient space.
\end{proof}

\begin{prop}\label{prop:LimitOfL} Suppose that $\{B_t\}_{t=1}^\infty$ is a sequence of maximal $S$-free sets such that $L_{B_t}+\lin(\conv(S)) = \R^n$, where $L_{B_t} = \{r : a_i^t\cdot r = a_j^t\cdot r, \;\quad \forall i,j \in I\}$. If $B_t\to B$, and $B \cap \conv(S)$ is a full dimensional polytope, then $L_B+\lin(\conv(S)) = \R^n$, where $L_B = \{r : a_i\cdot r = a_j\cdot r, \;\quad \forall i,j \in I\}$.
\end{prop}

\begin{proof} Suppose $\dim(\lin(\conv(S)))= k$. As $L_B\cap \lin(\conv(S)) = \{0\}$ from Proposition~\ref{prop:L-lin}, it is sufficient to show that $\dim(L_B) \geq  n-k$. Since $B_t \to B$, we have $B_t \cap \conv(S) \to B \cap \conv(S)$, and since $B\cap \conv(S)$ is a full dimensional polytope, by Propositions~\ref{prop:PolytopeBounded} and~\ref{prop:full-polytope} we eventually have that $B_t \cap \conv(S)$ is a polytope. Thus, $L_{B_t} \cap \lin(\conv(S)) = \{0\}$ by Proposition~\ref{prop:L-lin}. Since $L_{B_t}+\lin(\conv(S)) = \R^n$ for each $t$, $\dim(L_{B_t}) = n-k$. For each $i\neq j\in I$, define the matrix $A^t$ to have rows $a^t_i-a^t_j$ and $A$ to have the rows $a_i - a_j$. As $L_{B_t}=\nul(A^t)$, Proposition \ref{prop:NullConvergence} implies that $\dim(\nul(A)) \geq n-k$. Observing that $L_B=\nul (A)$ yields the desired result.
\end{proof}

\begin{theorem}\label{thm:LimitOfB}
Suppose $\{B_t\}_{t=1}^\infty$ is a sequence of maximal S-free sets possessing the covering property. If $B_t\to B$, where $B$ is a maximal S-free set and $B\cap \conv(S)$ is a polytope, then $B$ also possesses the covering property.
\end{theorem}

\begin{proof}
If $B$ is a half-space, then it is easy to check that $B$ has the covering property. Therefore, consider when $B$ is not a half-space and so $\intr(B\cap \conv(S))\neq \emptyset$ by Proposition~\ref{prop:basic-facts}.

From Proposition \ref{prop:PolytopeBounded} and~\ref{prop:full-polytope} we eventually have that $B_t \cap \conv(S)$ is a full-dimensional polytope. By Proposition~\ref{prop:ComplementarySpaces} we have $L_{B_t}+\lin(\conv(S)) = \R^n$. By Proposition~\ref{prop:LimitOfL}, $L_B+\lin(\conv(S)) = \R^n$. Moreover, since $B\cap \conv(S)$ is a polytope, we have $L_B \cap \lin(\conv(S)) = \{0\}$ by Proposition~\ref{prop:L-lin}. Define $M_t := R(S,B_t)\cap \lin(\conv(S))$ and $M:=R(S,B)\cap\lin(\conv(S))$. From Proposition~\ref{prop:CoveringEquivalence}, it is sufficient to show that $\lin(\conv(S))\subseteq M+W_S$.

Let $x\in\lin(\conv(S))$. Following Proposition \ref{prop:CoveringEquivalence}, for each $t$ there exists a spindle, $R_t(s_t)$, corresponding to $B_t$ such that $x\in D_t(s_t)+w_t$, where $D_t(s_t) =R_t(s_t)\cap\lin(\conv(S))$ and $w_t\in W_S$. We claim that $s_t$ and $w_t$ can be chosen independently of $t$.

\begin{proof}[Proof of claim:]
%Since $B_t\to B$, it also holds that $B_t\cap\conv(S)\to B\cap\conv(S)$. As $\intr(B\cap \conv(S))\neq\emptyset$, we may choose some $\omega\in \intr(B\cap \conv(S))$. Suppose $B\cap \conv(S) = \{x\in\R^n: Cx\leq b\}$, where $C, b$ are of the appropriate dimensions. Define $C_t$ similarly for each $t$ (note that defining a $b_t$ would result in $b_t=b$ for all $t$). As $\omega$ is contained in the interior, $C\omega < b$. Hence, there exists a $\delta>0$ such that $c_i\cdot\omega -b_i>\delta$ for each row in the system $Cx\leq b$. The rows of this system corresponding to $\conv(S)$ remain constant along the sequence $B_t\cap\conv(S)$, and so $c_i^t\cdot\omega -b_i = c_i\cdot \omega - b_i>\delta $ for all $t$. The rows of $C\omega< b$ corresponding to the constraints from $B$, are of the form $a_i\cdot \omega<b_i$. Rearranging and using the value of $\delta$, $a_i\cdot \omega-b_i > \delta$. As $a_i^t\to a_i$ for each $i\in I$, the inequality $a_i^t\cdot\omega -b_i>\delta$ must hold for all $i\in I$ for large enough $t$. Hence, $\omega\in \intr(B_t\cap \conv(S))$ for large $t$.

From Proposition \ref{prop:PolytopeBounded}, there exists a bounded set, $U$, that contains $B\cap \conv(S)$ and $B_t\cap \conv(S)$ for sufficiently large $t$. Consider the tail subsequence $\{B_t\}$ that has the property $B_t\cap \conv(S)\subseteq U$ for all $t$. As $U$ is bounded and $S$ is discrete, there is a finite number of points in $U\cap S$. Note that each spindle in $R(S, B_t)$ is anchored by a point in $B_t\cap S \subseteq U$. Therefore, there exists an $s\in S$ and a subsequence of $\{B_t\}$ such that $D_t(s_t)=D_t(s)$, for all $t$. Relabel such a subsequence by $\{B_t\}$.

Since the inner product is a continuous function on $\R^n$, $s \in B_t$ implies $s\in B$. Since $B_t\to B$, for a fixed $s$ it also follows that $D_t(s)\to D(s)$, where $D(s):=R(s)\cap \lin(\conv(S))$. As $L_{B_t}\cap \lin(\conv(S)) = \{0\}$ for each $t$, the set $D_t(s)$ is a polytope for each $t$. Similarly, $D(s)$ is a polytope. Again using Proposition \ref{prop:PolytopeBounded}, there exists a bounded set $V$ such that $D(s)\subseteq V$ and $D_t(s)\subseteq V$ for large $t$ (note that the origin is in each $D_t(s)$ and $D(s)$ and so the hypothesis of the Proposition \ref{prop:PolytopeBounded} is satisfied). In the same manner as above, for large $t$, $w_t\in D_t(s)-x \subseteq V-x$, which is a bounded set. Since $W_S = \lin(\conv(S) \cap \Lambda$ by Proposition~\ref{prop:W_S-characterization}, $W_S$ is discrete and there exists a $w\in W_S$ and a subsequence of $\{B_t\}$ (label this subsequence as $\{B_t\}$) such that $w_t=w$ for all $t$. Hence, $x\in D_t(s)+w$ for all $t$. \end{proof}

Since the inner product is a continuous function on $\R^n \times \R^n$, $x \in D_t(s) + w$ implies $x\in D(s)+w$. As $D(s) \subseteq M$, it follows that $x\in M+W_S$. Hence, $\lin(\conv(S))\subseteq M+W_S$, as desired.\end{proof}

\paragraph{The assumption that $B \cap \conv(S)$ is a polytope.} We end this section with a short justification of the assumption that $B \cap \conv(S)$ is a polytope that was made in Theorem~\ref{thm:LimitOfB}. Although it may seem restrictive at first, if $B\cap\conv(S)$ is not a polytope then one can reduce to that case in the following way. Let $N$ be the linear space spanned by $\rec(B \cap \conv(S))$. By Proposition~\ref{prop:basic-facts}(i), $N$ is a lattice subspace. Let $\bar B, \bar S, \bar \Lambda$ be the projection of $B, S, \Lambda$ onto the orthogonal subspace $N^\perp$ of $N$. By a well-known property of lattices, $\bar\Lambda$ is a lattice. Also, since $\conv(\bar S)$ is the projection of $\conv(S)$ and $S = \conv(S) \cap (b+ \Lambda)$, we have $\bar S = \conv(\bar S) \cap (\bar b + \bar\Lambda)$ where $\bar b$ is the projection of $b$. Hence, $\bar S$ is a polyhedrally-truncated affine lattice in $N^\perp$ and $\bar B$ is a maximal $\bar S$-free set. Moreover, $\bar B \cap \conv(\bar S)$ is a polytope, since $N$ is the linear space spanned by $\rec(B \cap \conv(S))$. Note that $N \subseteq L_B$ by Proposition~\ref{prop:basic-facts}(i), and by Proposition~\ref{prop:basic-facts}(ii), $R(S,B) = R(\bar S, \bar B) + N$. Hence, $B$ has the covering property with respect to $S$ if and only if $\bar B$ has the covering property with respect to $\bar S$. Therefore, to check if $B$ has the covering property with respect to $S$, one can check if $\bar B$ can be obtained as the limit of $\bar S$-free sets with the covering property. 

\section{Application: Iterative application of coproducts and limits}\label{sec:examples}

In this section, we show some examples demonstrating the versatility of the coproduct and limit operations to obtain new and interesting families of bodies with the covering property. We note that the coproduct operation is associative: $(C_1 \copr C_2) \copr C_3 = C_1 \copr (C_2 \copr C_3)$. Thus, we will use notation such as $C_1 \copr C_2 \copr \ldots \copr C_k$ without any ambiguity.

\begin{enumerate}
\item {\bf Crosspolytopes.} Let $a_1, \ldots, a_n\in \R$ and $b_1, \ldots, b_n \in \R$ such that $a_j < 0 < b_j$ for all $j=1, \ldots, n$ and $\sum_{j=1}^n\frac1{b_j- a_j} = 1$. Consider the set of $2n$ points $X = \{(0, \ldots, a_j, \ldots, 0), (0, \ldots, b_j, \ldots, 0) : j=1, \ldots, n\}$ where the nonzero entry is in coordinate $j$. Define $S = \Z^n + (\frac{b_1}{b_1 - a_1}, \ldots, \frac{b_n}{b_n - a_n})$. Then the crosspolytope $\conv(X)$ is a maximal $S$-free set with the covering property. 

This follows from the fact that $\conv(X) = (b_1 - a_1)I_1 \copr (b_2 - a_2)I_2 \copr \ldots \copr (b_n - a_n)I_n$ where $I_j$ is the interval $[\frac{a_j}{b_j - a_j}, \frac{b_j}{b_j - a_j}]$; $I_j$ is therefore a maximal $S_j$-free set with the covering property where $S_j = \Z + \frac{b_j}{b_j - a_j}$. Applying Theorem~\ref{thm:coproduct} shows that the crosspolytope $\conv(X)$ has the covering property.

\item {\bf Simplices.} Let $b_1, \ldots, b_n\in \R$ such that $0 < b_j$ for all $j=1, \ldots, n$ and $\sum_{j=1}^n\frac{1}{b_j}= 1$. Then the simplex $\conv\{0, b_1e^1, b_2e^2, \ldots, b_ne^n\}$, where the $e^i$ denotes the $i$-th unit vector in $\R^n$, is maximal $\Z^n$-free set with the covering property. This follows from taking the limit of the crosspolytopes defined in 1. above as $a_i \to 0$, and applying Theorem~\ref{thm:LimitOfB}. This generalizes the Type 1 triangle from the literature, as well as its higher dimensional analogue $\{0,ne^1, \ldots, ne^n\}$ that has been studied in~\cite{basu2012unique,ccz}, where this special case was shown to have the covering property using completely different arguments.

\item {\bf Further examples.} In three dimensions, one can show that there exist lattice-free sets with the covering property with 2,3,4,5,6, and 8 facets. By taking cylinders over the two-dimensional sets one can obtain 2,3, and 4 facets.  The crosspolytope from 1. above gives 8 facets. The coproduct of a triangle and an interval has 6 facets. Five facets can be obtained by taking the coproduct of a quadrilateral and an interval which gives a crosspolytope with 8 facets, and then taking a limit to reduce the number of facets from 8 to 5: four of the facets degenerate into a single facet. This can be iterated to generate bodies with the covering property in 4, 5, and any number of dimensions. 

We give another example of the kind of results one can prove using coproducts and limit operations. In $\R^k$ ($k\geq 2$), one can explicitly construct a maximal $\Z^k$-free set with $2^{k-1}+1$ facets with the covering property. This can be seen by taking the coproduct of $k$ intervals (to get the crosspolytope with $2^k$ facets) and then taking the limit to reduce $2^{k-1}$ of the facets into a single facet. We believe that the coproduct and limit tools could be useful in attacking questions of the following flavor:

\begin{question} For a fixed $n \in \N$, for which natural numbers in the range $2 \leq k \leq 2^n$ do there exist maximal lattice-free sets in $\R^n$ with $k$ facets that have the covering property?
\end{question}

Moreover, when considering $S$ of the form $\Z^n \times \Z^q_+$ one can construct unbounded polyhedra, by taking the coproduct of a translated cone in $\R^2$ (which has been shown in the literature to be a maximal $S$-free set with the covering property when $S$ is a translated lattice intersected by a halfspace) and quadrilaterals, triangles, and intervals (and iterating to get into arbitrarily high dimensions).

\end{enumerate}

We feel establishing the covering property of the examples above, or even discovering that these bodies have the covering property, would have been challenging without the tools of the coproduct and the limit operation. We mention that the constructions for the crosspolytopes and simplices above were first given in~\cite{averkov-basu-lifting}. The unbounded constructions in 3. above would not have been possible without the results of this current manuscript. Moreover, these operations are constructive and therefore potentially useful beyond purely theoretical questions about the covering property.

\section*{Acknowledgments} We are very grateful to two anonymous referees for insights that helped to considerably simplify and improve the proof of Theorem 3.1 from a previous version. Their suggestions also helped to present all the results in a more concise and effective manner.

\bibliographystyle{plain}
\bibliography{full-bib}

\appendix
\section{Appendix}
\begin{prop}\label{prop:minimality}
Let $S \subseteq \R^n \setminus \{0\}$ be nonempty. Every valid cut-generating pair for $S$ is dominated by a minimal cut-generating pair for $S$.
\end{prop}

\begin{proof} Fix $s^* \in S$ which is nonempty. Note that any cut-generating pair $(\psi, \pi)$ satisfies $\psi(r) + \psi(s^* - r) \geq 1$ and $\pi(r) + \pi(s^* - r) \geq 1$ for every $r \in \R^n$.

Let $(\bar\psi, \bar\pi)$ be a cut-generating pair. Define two new functions $\phi_1(r) = 1 - \bar\psi(s^* - r)$ and $\phi_2(r) = 1 - \bar\pi(s^* - r)$. Let $\mathcal{I}$ be the set of cut generating functions $(\psi, \pi)$ such that $\psi \leq \bar\psi$ and $\pi \leq \bar \pi$. Note that any element $(\psi, \pi) \in \mathcal{I}$ satisfies $\psi(r) \geq 1 - \psi(s^* - r) \geq 1 - \bar\psi(s^* - r) = \phi_1(r)$ and similarly, $\pi(r) \geq \phi_2(r)$.

We show that every chain in $\mathcal{I}$ has a lower bound in $\mathcal{I}$. Then by Zorn's lemma, $\mathcal{I}$ will contain a minimal element, proving the result.

Consider any chain $\mathcal{C}$ in $\mathcal{I}$. For any element $(\psi, \pi) \in \mathcal{C}$, we know that $\psi \geq \phi_1$ and $\pi \geq \phi_2$. Therefore, $\tilde{\psi}(r) := \inf\{\psi(r) : (\psi, \pi) \in \mathcal{C}\}$ and $\tilde{\pi}(r) := \inf\{\pi(r) : (\psi, \pi) \in \mathcal{C}\}$ are well-defined real-valued functions. It is easy to verify that $(\tilde{\psi}, \tilde{\pi})$ are cut-generating functions, and are therefore in $\mathcal{I}$. This completes the proof that each chain has a lower bound in $\mathcal{I}$.
\end{proof}

\begin{prop}\label{prop:minimality-lifting}
Let $S\subseteq \R^n\setminus \{0\}$ be nonempty and let $\psi$ be a cut-generating function for $S$. Every lifting $\pi$ for $\psi$ is dominated by a minimal lifting.
\end{prop}

\begin{proof}
 Fix $s^* \in S$ which is nonempty. For any lifting $\pi$ of $\psi$, we must have $\psi(s^* - r) + \pi(r) \geq 1$ and therefore, if we define $\phi(r) = 1 - \psi(s^* - r)$, we have that $\pi(r) \geq \phi(r)$. The proof idea of Proposition~\ref{prop:minimality} can again be used to show that every lifting is dominated by a minimal lifting.
\end{proof}

\begin{prop}\label{prop:periodic}
Let $S\subseteq \R^n\setminus \{0\}$ and let $\psi$ be a cut-generating function for $S$. Every minimal lifting of $\psi$ is periodic along $W_S$.
\end{prop}

\begin{proof} Let $\pi$ be a minimal lifting of $\psi$. Assume to the contrary that $\pi$ is not periodic along $W_S$. Therefore, there exists some $\hat{p}\in \R^n$ and $w\in W_S$ such that $\pi(\hat{p})\neq \pi(\hat{p}+w)$. Since $-w\in W_S$, we may assume $\pi(\hat{p})>\pi(\hat{p}+w)$. Define a function $\tilde{\pi}:\R^n\to \R$ by $\tilde{\pi}(p)=\pi(\hat{p}+w)$ if $p=\hat{p}$, and $\tilde{\pi}(p)=\pi(p)$ otherwise. If $\tilde{\pi}$ is a lifting of $\psi$, then we will have $\pi$ is not minimal, yielding a contradiction. Hence, it is sufficient to show that $\tilde\pi$ is a lifting of $\psi$.

Take $k,l\in \Z_+, R\in \R^{n\times k}$, and $P\in \R^{n\times l}$. We must show that $\eqref{psi pi ineq}$ holds for all $(s,y)\in X_S(R,P)$, so take $(s,y)\in X_S(R,P)$. Note that the columns of $P$ may be taken to be distinct by adding the components of $y$ that correspond to equal columns. Consider three cases.

%\begin{enumerate}
%\item[Case 1:] 
\emph{Case 1:} Suppose that $P$ does not contain $\hat{p}$ as one of its columns. Then
\begin{equation*}
\sum_{i=1}^k\psi(r_i)s_i + \sum_{j=1}^\ell\tilde{\pi}(p_j)y_j = \sum_{i=1}^k\psi(r_i)s_i + \sum_{j=1}^\ell\pi(p_j)y_j\ge 1,
\end{equation*}
where the inequality arises since $\pi$ is a lifting of $\psi$.

%\item[Case 2:] 
\emph{Case 2:} Suppose that $P$ contains $\hat{p}$ as one of its columns, but not $\hat{p}+w$. Let $P^o$ and $y^o$ be the columns and values of $P$ and $y$, respectively, that do not correspond to $\hat{p}$. Let $y_{\hat j}$ be the component of $y$ corresponding to $\hat p$. Using the definition of $W_S$ and the fact that $y_{\hat j} \in \Z_+$, it follows that
\begin{equation*}
Rs+Py = Rs+P^oy^o+\hat{p}y_{\hat{j}}\in S \iff Rs+P^oy^o+\hat{p}y_{\hat{j}}+wy_{\hat{j}} = Rs+P^oy^o+(\hat{p}+w)y_{\hat{j}} \in S.
\end{equation*}
If we define $P'\in \R^{n\times k}$ to be the columns of $P^o$ adjoined with $\hat{p}+w$, then the equivalence above implies
\begin{equation*}
\sum_{i=1}^k\psi(r_i)s_i + \sum_{j=1}^\ell\tilde{\pi}(p_j)y_j = \sum_{i=1}^k\psi(r_i)s_i + \sum_{j=1,j\neq \hat{j}}^\ell\pi(p_j)y_j + \pi(\hat{p}+w)y_{\hat{j}}\ge 1,
\end{equation*}
where the inequality arises since $\pi$ is a lifting of $\psi$ and we can apply the cut-generating pair $(\psi, \pi)$ to $(s, (y^o, y_{\hat j})) \in X_S(R,P')$.

%\item[Case 3:] 
\emph{Case 3:} Suppose that $P$ contains $\hat{p}$ and $\hat{p}+w$ as columns. Using a similar argument as above, define $P'$ to be the columns of $P$ without $\hat{p}$. This yields the same inequality as Case 2. \end{proof}

%\end{enumerate}\end{proof}

\begin{prop}\label{prop:psi-ast}
Let $S\subseteq \R^n\setminus \{0\}$ and let $\psi$ be a cut-generating function for $S$. If $R_\psi + W_S = \R^n$, then $\psi^\ast$ defined in~\eqref{eq:formula-for-lifting} is a minimal lifting and $\psi^\ast(r) = \psi(r + w)$ for any $w$ such that $r + w \in R_\psi$.
\end{prop}
\begin{proof}
It is not hard to verify that $\psi^\ast$ is a lifting of $\psi$. Consider any minimal lifting $\pi$. Consider any $r \in \R^n$ and let $w \in W_S$ such that $r + w \in R_\psi$. By Proposition~\ref{prop:periodic}, $\pi(r) = \pi(r+w) = \psi(r+w) \geq \psi^*(r)$. This implies that $\pi(r) = \psi^*(r) = \psi(r+w)$ since $\pi$ is a minimal lifting.\end{proof}

\begin{proof}[Proof of Proposition~\ref{prop:W-manipulation}] ~
\begin{itemize}
\item[(i)]  Let $M:\R^n\to\R^n$ be an invertible linear transformation and $m\in \R^n$. Note that
\begin{align*}
W_{M(S)+m} &= \left\{w\in \R^n: (M(s)+m)+\lambda w\in M(S)+m, \forall~ s \in S, ~\lambda\in \Z \right\}\\
& = \left\{w\in \R^n: M(s)+\lambda w\in M(S), \forall~ s \in S, ~\lambda\in \Z \right\}\\
& = \left\{w\in \R^n: s+\lambda M^{-1}(w)\in S, \forall~ s\in S, ~\lambda\in \Z \right\}\\
& = \left\{w\in \R^n: M^{-1}(w)\in W_S\right\}\\
& = \left\{w\in \R^n: w\in M(W_S)\right\}\\
& = M(W_S).
\end{align*}
%\item[(i)] Note that
%\begin{align*}
% x \in \mu W_s & \iff \frac{x}{\mu} \in W_S\\
% & \iff s+\lambda(\frac{x}{\mu}) \in S, ~\forall s\in S,~\forall \lambda\in\Z\\
% & \iff \mu s+\lambda x\in \mu S, ~\forall s\in S, \forall \lambda\in \Z\\
% & \iff s+\lambda x\in \mu S, ~\forall s\in \mu S, \lambda\in \Z\\
% &\iff x\in W_{\mu S}.
% \end{align*}
%\item[(ii)] Note that
%\begin{align*}
% x \in W_{S+t} &\iff (s+t)+\lambda x \in S+t,~\forall s\in S, ~\forall \lambda\in \Z\\
%&\iff s+\lambda x\in S, ~\forall s\in S, ~\forall \lambda\in\Z\\
%&\iff x\in W_S
%\end{align*}
\item[(ii)] Note that
\begin{align*}
 (x_1, x_2) \in W_{S_1\times S_2} & \iff (s_1+\lambda x_1, s_2+\lambda x_2) \in S_1\times S_2, ~\forall (s_1, s_2)\in S_1\times S_2, ~\forall \lambda\in \Z\\
& \iff s_i+\lambda x_i\in S_i, ~\forall i\in \{1,2\}, ~\forall s_i\in S_i, ~\forall \lambda\in \Z\\
& \iff (x_1, x_2) \in W_{S_1}\times W_{S_2}.
\end{align*}
\end{itemize}
\end{proof}

\end{document}